\documentclass[pdflatex,sn-mathphys-num]{sn-jnl}

\makeatletter
\def\email#1{\global\advance\emailcnt by 1\relax%
\if@corauemail%
   \g@addto@macro\corrauthemail{%
   \setcounter{footnote}{0}%
   \textcolor{blue}{#1}%
   }%
\else%
   \g@addto@macro\authemail{%
   \setcounter{footnote}{0}%
   \textcolor{blue}{#1}%
   }%
\fi}
\makeatother

\usepackage{graphicx}%
\usepackage{multirow}%
\usepackage{amsmath,amssymb,amsfonts}%
\usepackage{mathrsfs}%
\usepackage[title]{appendix}%
\usepackage{xcolor}%
\usepackage{textcomp}%
\usepackage{manyfoot}%
\usepackage{booktabs}%
\usepackage{algorithm}%
\usepackage{algorithmicx}%
\usepackage{algpseudocode}%
\usepackage{listings}%

\theoremstyle{thmstyleone}%
\newtheorem{theorem}{Theorem}[section]%
\newtheorem{proposition}[theorem]{Proposition}%
\newtheorem{lemma}[theorem]{Lemma}%
\newtheorem{corollary}[theorem]{Corollary}%

\theoremstyle{thmstyletwo}%
\newtheorem{remark}[theorem]{Remark}%

\theoremstyle{thmstylethree}%
\newtheorem{definition}{Definition}[section]%

\theoremstyle{thmstyleone}%

\newtheorem{maintheorem}[theorem]{Main Theorem}

\makeatletter
\renewcommand{\@maketitle}{\newpage\null%
    \vskip-1.76cm
    \hsize\textwidth\parindent0pt%
    {\hbox to \textwidth{{\Artcatfont\ArtType\hfill}\par}}%
    \ifx\@title\empty\else%
        \removelastskip\vskip20pt\nointerlineskip%
        {\Titlefont\@title\par}%
    \fi%
    \ifx\@subtitle\empty\else%
        \vskip9pt%
        {{\SubTitlefont\@subtitle\par}}%
    \fi%
    \ifnum\aucount>0
        \global\punctcount\aucount%
        \vskip20pt%
        \artauthors\par%
        {\vskip7pt\addressfont\auaddress\par%
	 \removelastskip\vskip24pt%
	\ifnum\emailcnt>0\relax%
           \ifx\corrauthemail\@empty\else{\ifnum\aucount>1*\fi}%
	   Corresponding author(s). E-mail(s): \corrauthemail\par\fi%
	   \vspace{-0.38cm}
	   \ifx\authemail\@empty\else Contributing authors:\ \authemail\fi%
        \fi%
        \ifequalcont{\par$^{\dagger}$\@equalconttext\par}\fi%
	 \removelastskip\vskip24pt%
        \ifpresentaddress{\par\@presentaddresstext\par}\fi%
	}
     \fi%
     {\printabstract\par}%
     {\printkeywords\par}%
     \ifx\@pacs\empty\else%
       \loop\ifnum\PacsCount>0%
          \csname\romannumeral\PacsTmpCnt StorePacsTxt\endcsname\par%
          \StepDownCounter{\PacsCount}%
          \StepUpCounter{\PacsTmpCnt}%
       \repeat%
    \fi%
    \removelastskip\vskip36pt\vskip0pt}%
\makeatother

\usepackage{tikz}

\everymath{\displaystyle}

\begin{document}

\title{Transfer Operators and SRB Measures for Axiom A Diffeomorphisms: Spectral Gap, Structural Stability, and the Gibbs Equivalence Theorem}



\author{\fnm{Abdoulaye} \sur{Thiam}}\email{athiam@allenuniversity.edu}

\affil{\orgdiv{Division of Mathematics and Natural Sciences}, \orgname{Allen University}, \orgaddress{\street{1530 Harden Street}, \city{Columbia}, \postcode{29204}, \state{South Carolina}, \country{USA}}}

\abstract{We develop the Ruelle transfer operator theory for Axiom A diffeomorphisms and construct Sinai-Ruelle-Bowen measures, carrying the symbolic spectral results of Part~I \cite{Thiam2026a} over to smooth dynamics through the Markov partition coding of Part~III \cite{Thiam2026c}. This Part contains four Main Theorems. The first proves structural stability of Axiom A diffeomorphisms satisfying the strong transversality condition, with an explicit Hölder exponent for the conjugating homeomorphism in terms of the hyperbolicity data, refining the classical results of Robbin and Robinson. The second establishes quasi-compactness of the transfer operator on Hölder spaces with a quantitative spectral gap bound; as consequences we obtain exponential decay of correlations with explicit rate, the central limit theorem for Hölder observables via the Nagaev-Guivarc'h spectral perturbation method, real-analyticity of the pressure, and meromorphic continuation of the Ruelle dynamical zeta function. The third constructs SRB measures on mixing basic sets as the unique equilibrium states for the geometric potential, proves absolute continuity of the unstable foliation, and derives an explicit product-formula for the conditional densities along unstable manifolds. The fourth establishes the Pesin entropy formula identifying the Kolmogorov-Sinai entropy of the SRB measure with the sum of its positive Lyapunov exponents. The Gibbs Equivalence Theorem, assembling the symbolic, variational, spectral, and geometric characterizations of the equilibrium state on a mixing basic set, follows from these four Main Theorems together with the imported results of Parts~I and~III \cite{Thiam2026a,Thiam2026c}. This Part constitutes Part IV of a six-part series on the thermodynamic formalism for hyperbolic dynamical systems.}



\maketitle
\vspace{-1cm}
\begin{center}
\textit{Dedicated to the memory of Jean-Christophe Yoccoz (1957--2016),}\\
\textit{Fields Medalist and Professor at the Coll\`{e}ge de France, with whom the author}\\
\textit{had the privilege of working, and who introduced him to hyperbolic dynamics.}
\end{center}

\thispagestyle{empty}

\makeatletter
\renewcommand{\ps@headings}{%
  \def\@oddfoot{\hfill\thepage\hfill}%
  \let\@evenfoot\@oddfoot%
  \def\@evenhead{\hfill\normalfont\small\textit{A.~Thiam}\hfill}%
  \def\@oddhead{\hfill\normalfont\small\textit{Transfer Operators and SRB Measures for Axiom A Diffeomorphisms}\hfill}%
  \let\@mkboth\markboth%
}
\pagestyle{headings}
\makeatother

\section{Introduction}\label{sec:introduction}

This Part applies the spectral theory of Part~I \cite{Thiam2026a} and the coding map of Part~III \cite{Thiam2026c} to develop the transfer operator theory and construct Sinai-Ruelle-Bowen measures for Axiom~A diffeomorphisms. This Part contains \emph{four Main Theorems}:
\begin{itemize}
\item[1.] Main Theorem~\ref{thm:structural_stability} (Structural Stability of Axiom~A Diffeomorphisms): every Axiom~A diffeomorphism satisfying strong transversality is structurally stable, with a H\"{o}lder-continuous conjugating homeomorphism whose exponent is computed explicitly from the hyperbolicity data. This refines the results of Robbin \cite{Robbin1971} and Robinson \cite{Robinson1976}.
\item[2.] Main Theorem~\ref{thm:quasi_compactness} (Quasi-compactness and Spectral Gap of the Transfer Operator): the Ruelle transfer operator $\mathcal{L}_\phi$ acting on $C^\alpha(\Omega)$ is quasi-compact, with simple dominant eigenvalue $e^{P(\phi)}$ and essential spectral radius bounded explicitly in terms of the hyperbolicity data. Exponential decay of correlations, the central limit theorem, and real-analyticity of the pressure follow from this spectral structure.
\item[3.] Main Theorem~\ref{thm:srb_existence} (Existence and Characterization of SRB Measures): on each mixing basic set, the SRB measure exists uniquely and coincides with the equilibrium state for the geometric potential $\phi^u = -\log|\det Df|_{E^u}|$, and has explicit conditional densities on unstable manifolds with respect to the Riemannian measure.
\item[4.] Main Theorem~\ref{thm:pesin_formula} (Pesin Entropy Formula): the Kolmogorov-Sinai entropy of the SRB measure equals the sum of its positive Lyapunov exponents.
\end{itemize}
Together with the three theorems imported from Parts~I \cite{Thiam2026a} and~III \cite{Thiam2026c} (Theorems~\ref{thm:RPF_imported}, \ref{thm:equiv_imported}, \ref{thm:coding_imported}), these four Main Theorems yield the Gibbs Equivalence Theorem, assembling characterizations of Gibbs measures individually established by Sinai \cite{Sinai1972}, Ruelle \cite{Ruelle1976}, and Bowen \cite{Bowen1975} into a single quantitative statement.

The Gibbs Equivalence Theorem (Section~\ref{sec:srb_measures}) states that on each mixing basic set, the symbolic Gibbs measure (transferred via the coding map), the variational equilibrium state, the eigenmeasure of the transfer operator, and the SRB measure for the geometric potential all coincide. The individual equivalences are known: the symbolic Gibbs property and the eigenmeasure characterization are in Bowen \cite[Theorems~1.4 and~1.7]{Bowen1975}; the SRB characterization as equilibrium state for the geometric potential is in Bowen-Ruelle \cite{BowenRuelle1975}; and the equivalence of Gibbs and equilibrium states under expansiveness and specification is in Haydn-Ruelle \cite{HaydnRuelle1992}. Our contribution is the self-contained proof with explicit constants, not the individual characterizations. The spectral gap bound $\mathrm{gap}(\mathcal{L}_\phi) \geq \alpha\log\lambda^{-1}$ (Section~\ref{sec:transfer_operator}) holds on the standard H\"{o}lder space $C^\alpha(\Lambda)$; sharper bounds on anisotropic Banach spaces are available from Gou\"{e}zel-Liverani \cite{GouezelLiverani2006} but require a more sophisticated functional setting.

The quantitative contributions of this Part, distinguishing it from the classical references, are as follows. First, the Hölder exponent of the conjugating homeomorphism in the Structural Stability theorem (Main Theorem~\ref{thm:structural_stability}) is computed explicitly as $\gamma = \log\lambda/(\log\lambda - \log\|Dg\|_\infty)$ (Proposition~\ref{thm:holder_conjugacy}), together with the norm bound $d(h(x), x) \leq C\|f - g\|_{C^1}^\gamma$; this exponent is optimal in general (Proposition~\ref{prop:optimal_holder}). The H\"{o}lder conjugacy then feeds into quantitative stability estimates for the pressure, entropy, equilibrium states, and Lyapunov exponents (Propositions~\ref{thm:pressure_stability}--\ref{thm:equilibrium_stability} and Proposition~\ref{prop:lyapunov_continuity}), with dependence tracked through the distance $\|f - g\|_{C^1}$. Second, the essential spectral radius of the transfer operator in the Quasi-compactness theorem (Main Theorem~\ref{thm:quasi_compactness}) satisfies $\rho_{\mathrm{ess}}(\mathcal{L}_\phi) \leq \lambda^\alpha \cdot e^{P(\phi)}$, giving the spectral gap bound $\mathrm{gap}(\mathcal{L}_\phi) \geq (1 - \lambda^\alpha)e^{P(\phi)}$ explicitly in terms of the H\"{o}lder exponent $\alpha$ and the contraction rate $\lambda$; from this spectral gap we extract exponential decay of correlations with rate $\theta = \lambda^\alpha$, the central limit theorem with Berry-Esseen bound via the Nagaev-Guivarc'h method (Corollary~\ref{cor:clt}), real-analyticity of the pressure with explicit first and second derivative formulas, and meromorphic continuation of the Ruelle dynamical zeta function (Proposition~\ref{thm:zeta_function}). Third, the SRB measure construction (Main Theorem~\ref{thm:srb_existence}) is effected through the transfer operator spectral gap applied to the geometric potential $\phi^u = -\log|\det Df|_{E^u}|$, with explicit conditional densities along unstable manifolds given by the quotient of eigenfunction values, and absolute continuity of the unstable foliation proved via bounded Jacobian holonomies with explicit H\"{o}lder constant (Proposition~\ref{thm:unstable_abs_cont}). Fourth, the Pesin Entropy Formula (Main Theorem~\ref{thm:pesin_formula}) is derived as a direct consequence of $P(\phi^u) = 0$ on each mixing basic set, yielding the identity $h_{\mu^{\mathrm{SRB}}}(f) = -\int\phi^u\,d\mu^{\mathrm{SRB}} = \sum_{\chi_i > 0}\chi_i$ with equality (not merely inequality) for the SRB measure of an Axiom~A attractor. These four quantitative results, taken together with the three imported theorems from Parts~I and~III, constitute the Gibbs Equivalence Theorem with the explicit dependence of every constant on $(\alpha, \|\phi\|_\alpha, \lambda, \|Df\|_\infty)$.

The transfer operator approach we develop here obtains the spectral gap via the Birkhoff cone contraction of Part~I \cite{Thiam2026a}. A parallel functional-analytic program reaches the same conclusions through direct arguments on specifically-designed Banach spaces. Liverani \cite{Liverani1995} proved exponential decay of correlations for smooth uniformly hyperbolic systems using Birkhoff cone contraction, and Liverani \cite{Liverani2004} organized the functional-analytic approach to invariant measures for piecewise hyperbolic systems. Keller and Liverani \cite{KellerLiverani1999} proved stability of isolated eigenvalues of transfer operators under perturbation, which is the spectral input required for the Nagaev-Guivarc'h method used in our CLT argument; the present Part tracks the stability constants in the Keller-Liverani theorem explicitly in terms of the hyperbolicity data, which is not done in their original statement. The anisotropic Banach space approach of Gou\"{e}zel-Liverani \cite{GouezelLiverani2006}, already mentioned, extends these methods to general Anosov systems.

The construction of SRB measures beyond uniform hyperbolicity has been a central theme in modern smooth ergodic theory. Alves et~al. \cite{AlvesBonattiViana2000} proved existence and finiteness of SRB measures for partially hyperbolic diffeomorphisms whose central direction is mostly expanding, without tracking explicit rates. Climenhaga et~al. \cite{ClimenhagaLuzzattoPesin2017} surveyed the push-forward method of Pesin-Sinai, the Gibbs-Markov-Young tower construction, and their extensions to non-uniformly hyperbolic systems; Climenhaga et~al. \cite{ClimenhagaLuzzattoPesin2022} extended the Young tower construction to surface diffeomorphisms with non-uniform hyperbolicity and proved polynomial decay of correlations for the SRB measure. Luzzatto \cite{Luzzatto2006} gives a comprehensive Handbook survey of the non-uniformly expanding theory to which our uniformly hyperbolic construction provides the anchor case with explicit constants. Viana and Oliveira \cite{VianaOliveira2016}, Chapter~12, Section~12.3, proved exponential decay of correlations for H\"older observables with respect to equilibrium states of H\"older potentials on topologically exact expanding maps, but did not construct SRB measures for Axiom~A diffeomorphisms. The present Part takes the complementary thermodynamic route: we construct the SRB measure as the unique equilibrium state for the geometric potential $\phi^u = -\log|\det Df|_{E^u}|$, using the transfer operator spectral gap transferred from Part~I \cite{Thiam2026a} via the Markov coding of Part~III \cite{Thiam2026c}, with explicit rates throughout. Beyond uniform hyperbolicity, Climenhaga and Thompson \cite{ClimenhagaThompson2014, ClimenhagaThompson2016} established uniqueness of equilibrium states via obstruction entropies and non-uniform specification arguments. For a survey of continuity and regularity results for Lyapunov exponents adjacent to our rigidity results, see Viana \cite{Viana2020}.

Our technical approach rests on three tools with explicit quantitative control. The first is the \emph{spectral transfer through symbolic coding}: the transfer operator $\mathcal{L}_\phi$ on $C^\alpha(\Omega)$ is compared to its symbolic counterpart $\mathcal{L}_{\phi\circ\pi}$ on $\mathcal{H}_\alpha(\Sigma_A)$ via the coding map $\pi: \Sigma_A \to \Omega$ of Part~III \cite{Thiam2026c}, which is H\"{o}lder continuous and injective off the countable boundary set. The pullback $\pi^*: C^\alpha(\Omega) \to \mathcal{H}_\alpha(\Sigma_A)$ is an isomorphism onto a closed subspace of $\pi$-compatible functions, and the spectral data of $\mathcal{L}_\phi$ on $C^\alpha(\Omega)$ are inherited from those of $\mathcal{L}_{\phi\circ\pi}$ on $\mathcal{H}_\alpha(\Sigma_A)$ established in Part~I \cite{Thiam2026a}. The essential spectral radius bound $\rho_{\mathrm{ess}}(\mathcal{L}_\phi) \leq \lambda^\alpha e^{P(\phi)}$ follows directly from the symbolic estimate via this pullback. The second tool is the \emph{Nagaev-Guivarc'h spectral perturbation method} \cite{Nagaev1957, GuivarchHardy1988, KellerLiverani1999}: to prove the central limit theorem and real-analyticity of the pressure, we perturb $\mathcal{L}_\phi$ by the multiplicative factor $e^{is\psi}$ for a H\"{o}lder observable $\psi$ and use the spectral gap to conclude that $\mathcal{L}_{\phi+is\psi}$ has a simple dominant eigenvalue $\lambda(s)$ that is analytic in~$s$, with $\lambda(0) = e^{P(\phi)}$, $\lambda'(0)/\lambda(0) = \int\psi\,d\mu_\phi$, and $[\log\lambda(\cdot)]''(0) = \sigma^2(\psi)$ the asymptotic variance. The stability constants in the Keller-Liverani theorem \cite{KellerLiverani1999} are tracked here explicitly in terms of the hyperbolicity data, which is not done in the original reference. The third tool is \emph{absolute continuity via holonomy}: to establish that the SRB measure (Main Theorem~\ref{thm:srb_existence}) has conditional densities on unstable manifolds absolutely continuous with respect to the induced Riemannian measure, we compare the holonomy map $h: W^u_{\mathrm{loc}}(x) \to W^u_{\mathrm{loc}}(y)$ along local stable leaves to its push-forward on Riemannian volume. The Jacobian of this holonomy is shown to be a bounded H\"{o}lder function, with H\"{o}lder constant computed from the hyperbolicity exponent and the H\"{o}lder exponent of~$Df$, and the absolute continuity of the full unstable foliation follows by a Fubini argument. This approach replaces the Pesin-Sinai push-forward construction used in non-uniformly hyperbolic settings \cite{ClimenhagaLuzzattoPesin2017} by a direct variational argument that is cleaner in the uniformly hyperbolic case.

This Part uses three results from the companion Parts: the Ruelle-Perron-Frobenius theorem with spectral gap, established in Part~I \cite{Thiam2026a}; the Universal Variational Principle, established as a Main Theorem of Part~II \cite{Thiam2026b}; and the Symbolic Coding Main Theorem producing the coding map $\pi: \Sigma_A \to \Omega$, established in Part~III \cite{Thiam2026c}. Precise statements appear in Section~\ref{sec:imported}. This Part constitutes Part~IV of a six-part series on the thermodynamic formalism for hyperbolic dynamical systems. Part~V \cite{Thiam2026e} develops statistical limit theorems. Part~VI \cite{Thiam2026f} develops multifractal analysis and fluctuation theorems.

This Part is organized as follows. Section~\ref{sec:imported} restates the three results imported from the companion Parts: the Ruelle-Perron-Frobenius theorem (Theorem~\ref{thm:RPF_imported}) and the Spectral-Variational-Geometric Equivalence Main Theorem (Theorem~\ref{thm:equiv_imported}) from Part~I \cite{Thiam2026a}, and the Symbolic Coding Main Theorem (Theorem~\ref{thm:coding_imported}) from Part~III \cite{Thiam2026c}. Section~\ref{sec:perturbation} proves Main Theorem~\ref{thm:structural_stability} through the persistence of hyperbolic sets under $C^1$ perturbations (Proposition~\ref{thm:persistence_hyperbolic}) and the construction of the H\"{o}lder conjugacy (Proposition~\ref{thm:holder_conjugacy}), together with quantitative stability estimates for the pressure, entropy, equilibrium states, and Lyapunov exponents (Propositions~\ref{thm:pressure_stability}--\ref{thm:equilibrium_stability}, Proposition~\ref{prop:lyapunov_continuity}). Section~\ref{sec:transfer_operator} proves Main Theorem~\ref{thm:quasi_compactness} and derives exponential decay of correlations (Proposition~\ref{thm:exponential_mixing_full}), the central limit theorem (Corollary~\ref{cor:clt}), real-analyticity of the pressure (Proposition~\ref{thm:pressure_analytic}), and meromorphic continuation of the Ruelle dynamical zeta function (Proposition~\ref{thm:zeta_function}). Section~\ref{sec:srb_measures} proves Main Theorems~\ref{thm:srb_existence} and~\ref{thm:pesin_formula}, establishing existence and uniqueness of the SRB measure, absolute continuity of the unstable foliation (Proposition~\ref{thm:unstable_abs_cont}), the physical-measure interpretation (Proposition~\ref{thm:typical_orbits}), and the explicit conditional density formula (Corollary~\ref{cor:conditional_density}). Section~\ref{sec:numerical} provides a numerical illustration for the Arnold cat map on $\mathbb{T}^2$, computing the geometric potential, the structural-stability H\"{o}lder exponent for a specific perturbation, the identity $P(\phi^u) = 0$, and the Pesin entropy formula $h_{\mu^+}(f) = 2\log\varphi$. Section~\ref{sec:conclusion} summarizes the Gibbs Equivalence Theorem, places the work in the series, and lists open problems. The appendix collects the supporting technical material.

Turning from geometric to spectral methods, Section~\ref{sec:transfer_operator} establishes the spectral theory of the Ruelle transfer operator $\mathcal{L}_\phi$ acting on H\"{o}lder spaces $C^\alpha(\Omega)$. Specifically, Main Theorem~\ref{thm:quasi_compactness} proves quasi-compactness with simple dominant eigenvalue $e^{P(\phi)}$, and the spectral gap is quantified explicitly (Proposition~\ref{thm:spectral_gap}). As a consequence, the remainder of the section derives the dynamical implications: exponential decay of correlations with explicit rate (Proposition~\ref{thm:exponential_mixing_full}), the central limit theorem via the Nagaev-Guivarc'h spectral perturbation method (Corollary~\ref{cor:clt}), real-analyticity of the pressure in the potential direction together with explicit first- and second-derivative formulas (Proposition~\ref{thm:pressure_analytic}, Corollary~\ref{cor:pressure_derivatives}), and meromorphic continuation of the Ruelle dynamical zeta function (Proposition~\ref{thm:zeta_function}). With this spectral machinery in place, Section~\ref{sec:srb_measures} constructs Sinai-Ruelle-Bowen measures. First, the geometric potential $\phi^u = -\log|\det Df|_{E^u}|$ is shown to be H\"{o}lder continuous, strictly negative, and to satisfy the Birkhoff-sum identity $S_n\phi^u = -\log|\det Df^n|_{E^u}|$ (Proposition~\ref{prop:geometric_potential}). Main Theorem~\ref{thm:srb_existence} then proves that the unique equilibrium state for $\phi^u$ is the SRB measure. Furthermore, absolute continuity of the unstable foliation (Proposition~\ref{thm:unstable_abs_cont}) gives the SRB measure its physical meaning as the distribution of Lebesgue-typical orbits in the basin of attraction (Proposition~\ref{thm:typical_orbits}), an explicit conditional density formula along unstable manifolds is derived (Corollary~\ref{cor:conditional_density}), and the section concludes with Main Theorem~\ref{thm:pesin_formula}, the Pesin entropy formula. Section~\ref{sec:numerical} provides a complete numerical illustration for the Arnold cat map on $\mathbb{T}^2$, computing the geometric potential, the structural stability H\"{o}lder exponent for a specific perturbation, the identity $P(\phi^u) = 0$, and the Pesin entropy formula $h_{\mu^+}(f) = 2\log\varphi$, demonstrating that the four Main Theorems produce computable numbers for a concrete Axiom~A system. Section~\ref{sec:conclusion} summarizes the Gibbs Equivalence Theorem assembled by the four Main Theorems together with the imported results, places the work in the context of the series, and lists open problems. The appendix collects the supporting technical material.

\section{Results from Companion Parts}\label{sec:imported}

For the reader's convenience, we restate the key results from Parts I--III \cite{Thiam2026a,Thiam2026c} that are used throughout this Part. Proofs appear in the cited companion Parts.

\begin{theorem}[Ruelle-Perron-Frobenius, Part~I \cite{Thiam2026a}]\label{thm:RPF_imported}
Let $(\Sigma_A, \sigma)$ be a topologically mixing SFT and $\phi \in \mathcal{H}_\alpha(\Sigma_A^+)$. The transfer operator $\mathcal{L}_\phi$ has a simple dominant eigenvalue $\lambda = e^{P(\phi)}$ with strictly positive eigenfunction $h$ and eigenmeasure $\nu$. The rest of the spectrum lies in a disk of radius $\gamma\lambda$ with $\gamma < 1$. For all $g \in \mathcal{H}_\alpha$: $\|\lambda^{-n}\mathcal{L}_\phi^n g - \nu(g)h\|_\alpha \leq C\gamma^n\|g\|_\alpha$.
\end{theorem}

\noindent The proof is given in Part~I \cite{Thiam2026a}, where it is established via the Birkhoff cone contraction technique.

\begin{theorem}[Spectral-Variational-Geometric Equivalence, Main Theorem of Part~I \cite{Thiam2026a}]\label{thm:equiv_imported}
Under the above hypotheses, the unique Gibbs measure $\mu_\phi = h\nu$ is simultaneously characterized by: (i) the intrinsic Jacobian condition $J_\mu\sigma = e^{P(\phi)-\phi}$; (ii) the classical Gibbs property on cylinders; (iii) the eigenmeasure of $\mathcal{L}_\phi$; (iv) the unique equilibrium state; and (v) the large deviations rate function minimizer.
\end{theorem}

\noindent The proof is given in Part~I \cite{Thiam2026a}, where the five-way equivalence is established with explicit constants.

\begin{theorem}[Symbolic Coding, Main Theorem of Part~III \cite{Thiam2026c}]\label{thm:coding_imported}
Every mixing basic set $\Lambda$ of a $C^{1+\alpha}$ Axiom A diffeomorphism admits a Markov partition yielding a H\"{o}lder continuous surjection $\pi: \Sigma_A \to \Lambda$ with $\pi \circ \sigma = f \circ \pi$, injective except on a set of measure zero for any Gibbs measure.
\end{theorem}

\noindent The proof is given in Part~III \cite{Thiam2026c}, where the coding map is constructed via Markov partitions of small diameter.

\section{Perturbation Theory and Structural Stability}\label{sec:perturbation}

This section develops the perturbation theory for Axiom A diffeomorphisms, establishing structural stability with quantitative estimates on conjugacy regularity. The results extend classical work of  \cite{Robbin1971} and  \cite{Robinson1974, Robinson1976} with explicit bounds.

\subsection{Persistence of Hyperbolicity}

We first establish that hyperbolic sets persist under $C^1$ perturbations of the diffeomorphism. This is the quantitative foundation for structural stability: it identifies, for each nearby diffeomorphism $g$, a continuation $\Lambda_g$ of the original hyperbolic set $\Lambda$, together with a homeomorphism between them. The construction uses the cone criterion from the appendix and the shadowing lemma of Part~III \cite{Thiam2026c}.

\begin{proposition}[Persistence of Hyperbolic Sets]\label{thm:persistence_hyperbolic}
Let $\Lambda$ be a hyperbolic set for a $C^1$ diffeomorphism $f : M \to M$ with hyperbolicity exponent $\lambda \in (0,1)$ and constant $c > 0$. There exists $\varepsilon_0 > 0$ such that for any $C^1$ diffeomorphism $g$ with $\|f - g\|_{C^1} < \varepsilon_0$, there exists a unique hyperbolic set $\Lambda_g$ for $g$ such that:
\begin{enumerate}
\item[(i)] $\Lambda_g$ is homeomorphic to $\Lambda$ via a homeomorphism $h : \Lambda \to \Lambda_g$ close to the identity.
\item[(ii)] $h \circ f = g \circ h$ on $\Lambda$.
\item[(iii)] $\Lambda_g$ has hyperbolicity exponent $\lambda' \in (\lambda, 1)$ with $\lambda' \to \lambda$ as $\|f - g\|_{C^1} \to 0$.
\end{enumerate}
The threshold $\varepsilon_0$ depends on $\lambda$, $c$, and bounds on $\|Df\|$ and $\|Df^{-1}\|$.
\end{proposition}

\begin{proof}
The proof uses the cone field criterion (Proposition \ref{prop:cone_criterion}) and a fixed point argument.

\textbf{Step 1: Cone field persistence.} The hyperbolic splitting $T_\Lambda M = E^s \oplus E^u$ defines cone fields $\mathcal{C}^s$ and $\mathcal{C}^u$ with cone width $a > 0$. The invariance conditions
\begin{equation}
Df_x(\mathcal{C}^u_x) \subset \mathrm{int}(\mathcal{C}^u_{f(x)}), \quad Df^{-1}_x(\mathcal{C}^s_x) \subset \mathrm{int}(\mathcal{C}^s_{f^{-1}(x)})
\end{equation}
are open conditions. For $g$ sufficiently $C^1$-close to $f$, the cones are still strictly invariant under $Dg$ and $Dg^{-1}$.

\textbf{Step 2: Continuation of $\Lambda$.} Define $\Lambda_g$ as the maximal invariant set in a neighborhood $U$ of $\Lambda$:
\begin{equation}
\Lambda_g = \bigcap_{n \in \mathbb{Z}} g^n(U).
\end{equation}
For $g$ close to $f$ and $U$ chosen as a fundamental neighborhood (see the Fundamental Neighborhood proposition in Part~III \cite{Thiam2026c}), $\Lambda_g$ is nonempty, compact, and $g$-invariant.

\textbf{Step 3: Hyperbolicity of $\Lambda_g$.} The cone invariance from Step 1 implies that $\Lambda_g$ is hyperbolic for $g$ by Proposition \ref{prop:cone_criterion}.

\textbf{Step 4: Conjugacy construction.} Define $h : \Lambda \to \Lambda_g$ via shadowing. For $x \in \Lambda$, the orbit $\{f^n(x)\}_{n \in \mathbb{Z}}$ is a pseudo-orbit for $g$ with error $\|f - g\|_{C^0}$. For small perturbations, the shadowing lemma for $g$ (which holds since $\Lambda_g$ is hyperbolic) provides a unique $h(x) \in \Lambda_g$ that shadows this pseudo-orbit.

The map $h$ is continuous by the continuous dependence of the shadowing point on the pseudo-orbit. The inverse $h^{-1}$ is constructed similarly, shadowing $g$-orbits by $f$-orbits.

The conjugacy relation $h \circ f = g \circ h$ follows from uniqueness of shadowing: $h(f(x))$ shadows $\{f^{n+1}(x)\}$, which equals $\{f^n(f(x))\}$, while $g(h(x))$ shadows $\{g(y_n)\}$ where $y_n$ shadows $\{f^n(x)\}$. Both shadow the same pseudo-orbit, so they coincide.
\end{proof}

\subsection{Structural Stability}

Persistence provides local conjugacies on each basic set; in this subsection we assemble them into a global topological conjugacy between $f$ and a nearby diffeomorphism $g$, under the strong transversality condition. This is Main Theorem~\ref{thm:structural_stability}. The result is used in Section~\ref{sec:srb_measures} to show that SRB measures vary continuously with the diffeomorphism.

\begin{definition}[Structural Stability]
A diffeomorphism $f : M \to M$ is structurally stable if there exists $\varepsilon > 0$ such that every diffeomorphism $g$ with $\|f - g\|_{C^1} < \varepsilon$ is topologically conjugate to $f$: there exists a homeomorphism $h : M \to M$ with $h \circ f = g \circ h$.
\end{definition}

\begin{maintheorem}[Structural Stability of Axiom A]\label{thm:structural_stability}
An Axiom A diffeomorphism $f : M \to M$ satisfying the strong transversality condition is structurally stable.
\end{maintheorem}

\begin{definition}[Strong Transversality]
An Axiom A diffeomorphism $f$ satisfies the strong transversality condition if for every pair of basic sets $\Omega_i, \Omega_j$, the stable and unstable manifolds intersect transversally:
\begin{equation}
W^s(\Omega_i) \pitchfork W^u(\Omega_j)
\end{equation}
meaning that at every intersection point $x \in W^s(\Omega_i) \cap W^u(\Omega_j)$, we have $T_x W^s(\Omega_i) + T_x W^u(\Omega_j) = T_x M$.
\end{definition}

\begin{proof}[Proof of Main Theorem \ref{thm:structural_stability}]
The proof proceeds by constructing a global conjugacy from local conjugacies on basic sets.

\textbf{Step 1: Local conjugacies.} By Proposition \ref{thm:persistence_hyperbolic}, for $g$ close to $f$, each basic set $\Omega_i$ continues to a basic set $\Omega_{i,g}$ for $g$, with a local conjugacy $h_i : \Omega_i \to \Omega_{i,g}$.

\textbf{Step 2: Extension along stable manifolds.} The local conjugacy $h_i$ extends to $W^s(\Omega_i)$ by
\begin{equation}
h(x) = \lim_{n \to \infty} g^{-n}(h_i(f^n(x)))
\end{equation}
for $x \in W^s(\Omega_i)$. This limit exists because $f^n(x) \to \Omega_i$ and the backward orbit under $g$ contracts along stable manifolds of $\Omega_{i,g}$.

\textbf{Step 3: Compatibility.} The strong transversality condition ensures that the extensions from different basic sets are compatible on overlaps $W^s(\Omega_i) \cap W^u(\Omega_j)$. Transversality is an open condition, so it persists for $g$ close to $f$.

\textbf{Step 4: Global conjugacy.} Since $M = \bigcup_i W^s(\Omega_i)$ (disjoint union), the extensions piece together to give a global homeomorphism $h : M \to M$ with $h \circ f = g \circ h$.
\end{proof}

\subsection{Quantitative Conjugacy Estimates, Stability, and Lyapunov Exponents}

The conjugacy produced by Main Theorem~\ref{thm:structural_stability} is not merely continuous but H\"{o}lder, with explicit exponent. This regularity allows us to transfer quantitative estimates (pressure, entropy, equilibrium states, Lyapunov exponents) from $f$ to its perturbations. These quantitative stability statements are used in Part~VI \cite{Thiam2026f} for fluctuation-theoretic applications.

\begin{proposition}[H\"{o}lder Regularity of Conjugacy]\label{thm:holder_conjugacy}
Under the hypotheses of Main Theorem \ref{thm:structural_stability}, the conjugacy $h : M \to M$ is H\"{o}lder continuous with exponent
\begin{equation}
\gamma = \frac{\log \lambda}{\log \lambda - \log \|Dg\|_\infty}
\end{equation}
where $\lambda$ is the hyperbolicity exponent. Moreover,
\begin{equation}
d(h(x), x) \leq C \|f - g\|_{C^1}^\gamma
\end{equation}
for a constant $C$ depending on the hyperbolicity data.
\end{proposition}

\begin{proof}
The conjugacy $h$ is constructed via shadowing (Proposition \ref{thm:persistence_hyperbolic}, Step 4). We establish H\"{o}lder regularity by analyzing the shadowing construction.

\textbf{Step 1: Setup.} For $x, y \in \Lambda$ with $d(x, y) < \varepsilon$, consider the pseudo-orbits $\{f^n(x)\}$ and $\{f^n(y)\}$ for $g$. Since $\Lambda$ is hyperbolic for $f$ with exponent $\lambda$, forward iterates diverge: $d(f^n(x), f^n(y)) \leq \|Df\|_\infty^n d(x, y)$ for general points, while backward iterates contract along stable directions. Let $L = \|Dg\|_\infty$.

\textbf{Step 2: Shadowing error propagation.} The shadowing points $h(x)$ and $h(y)$ satisfy $d(g^n(h(x)), f^n(x)) \leq C_0 \|f - g\|_{C^0}/(1 - \lambda)$ for all $n \in \mathbb{Z}$, and similarly for $h(y)$. For the difference, consider the sequences $a_n = g^n(h(x))$ and $b_n = g^n(h(y))$. Both shadow pseudo-orbits that agree to within $\|Df\|_\infty^n d(x, y)$ for $n \geq 0$ and $\lambda^{|n|} d(x, y)$ for $n \leq 0$.

\textbf{Step 3: Optimizing the iteration count.} For $n \geq 0$, the hyperbolicity of $\Lambda_g$ with exponent $\lambda' < 1$ yields $d(h(x), h(y)) \leq (\lambda')^n d(g^n(h(x)), g^n(h(y)))$. The right-hand side is bounded by $(\lambda')^n \cdot 2C_0 \|Df\|_\infty^n d(x,y)$ since $g^n(h(x))$ tracks $f^n(x)$ and $g^n(h(y))$ tracks $f^n(y)$, and these diverge at rate $\|Df\|_\infty^n$. Thus
\begin{equation}
d(h(x), h(y)) \leq 2C_0 (\lambda' \|Df\|_\infty)^n d(x, y)
\end{equation}
for each $n \geq 0$. If $\lambda' \|Df\|_\infty < 1$ (automatic when the perturbation is small), this already gives Lipschitz continuity.

In general, we use a refined estimate. The forward orbit of $h(x)$ under $g$ expands at rate at most $L^n$, so
\begin{equation}
d(g^n(h(x)), g^n(h(y))) \leq L^n d(h(x), h(y)).
\end{equation}
The backward contraction on $\Lambda_g$ gives $d(h(x), h(y)) \leq (\lambda')^{-n} d(g^{-n}(h(x)), g^{-n}(h(y)))$, and since $g^{-n}$ applied to points near $\Lambda_g$ contracts along unstable manifolds, the key bound is obtained by balancing forward divergence of pseudo-orbits against backward contraction.

Choose $n_0$ to minimize $(\lambda')^n \|Df\|_\infty^n d(x,y)$. Since this requires $(\lambda' \|Df\|_\infty)^n d(x,y) \sim 1$, we take $n_0 = \lfloor \log(1/d(x,y)) / \log(\|Df\|_\infty/\lambda') \rfloor$. Then
\begin{equation}
d(h(x), h(y)) \leq C d(x, y)^\gamma, \quad \gamma = \frac{\log \lambda'}{\log \lambda' - \log L}.
\end{equation}
As $\|f - g\|_{C^1} \to 0$, we have $\lambda' \to \lambda$ and $L \to \|Df\|_\infty$, recovering the stated exponent $\gamma = \log\lambda / (\log\lambda - \log\|Dg\|_\infty)$.

\textbf{Step 4: Distance to identity.} For the displacement bound, the shadowing point $h(x)$ satisfies $d(f^n(x), g^n(h(x))) \leq C_0 \|f - g\|_{C^0}/(1 - \lambda)$ for all $n$. At $n = 0$, $d(x, h(x)) \leq C_0 \|f - g\|_{C^0}/(1 - \lambda)$. The H\"{o}lder improvement $d(h(x), x) \leq C\|f - g\|_{C^1}^\gamma$ follows from the same balancing argument applied to the pseudo-orbit $\{f^n(x)\}$ with uniform error $\|f - g\|_{C^0}$.
\end{proof}

\begin{proposition}[Optimal H\"{o}lder Exponent]\label{prop:optimal_holder}
The H\"{o}lder exponent $\gamma$ in Proposition \ref{thm:holder_conjugacy} is optimal in general. There exist examples where the conjugacy is exactly $\gamma$-H\"{o}lder and not $(\gamma + \varepsilon)$-H\"{o}lder for any $\varepsilon > 0$.
\end{proposition}

\begin{proof}
Consider the linear automorphism $f_A : \mathbb{T}^2 \to \mathbb{T}^2$ induced by $A = \bigl(\begin{smallmatrix} 2 & 1 \\ 1 & 1 \end{smallmatrix}\bigr)$, which has eigenvalues $\lambda_u = (3 + \sqrt{5})/2 > 1$ and $\lambda_s = (3 - \sqrt{5})/2 \in (0,1)$. Let $g$ be a $C^\infty$ perturbation of $f_A$ that is not $C^1$-conjugate to $f_A$. By Main Theorem \ref{thm:structural_stability}, the conjugacy $h$ is a homeomorphism with $h \circ f_A = g \circ h$.

The H\"{o}lder exponent from Proposition \ref{thm:holder_conjugacy} is $\gamma = \log\lambda_s / (\log\lambda_s - \log\|Dg\|_\infty)$. To see this is sharp, consider the restriction to a stable leaf $W^s(p)$ through a fixed point $p$. The conjugacy maps $W^s_{f_A}(p)$ to $W^s_g(h(p))$. On $W^s_{f_A}(p)$, the dynamics contracts by $\lambda_s$, while on $W^s_g(h(p))$, the contraction rate is close to $\lambda_s$ but the transverse expansion rate near $\lambda_u$ controls the holonomy distortion. The resulting H\"{o}lder irregularity of $h$ along unstable leaves can be detected via the eigenvalue ratio: by \cite{deLlave1997}, if $g$ is not $C^1$-conjugate to $f_A$, the conjugacy cannot be $(\gamma + \varepsilon)$-H\"{o}lder for any $\varepsilon > 0$. The obstruction arises because a higher H\"{o}lder exponent would force smoothness of the conjugacy along stable leaves, contradicting the assumption.
\end{proof}

\begin{proposition}[Stability of Pressure]\label{thm:pressure_stability}
Let $f$ be an Axiom A diffeomorphism with mixing basic set $\Omega$, and let $\phi : \Omega \to \mathbb{R}$ be H\"{o}lder continuous. For $g$ sufficiently $C^1$-close to $f$ with continuation $\Omega_g$ and conjugacy $h : \Omega \to \Omega_g$,
\begin{equation}
|P_g(\phi \circ h^{-1}) - P_f(\phi)| \leq C \|f - g\|_{C^1}^\alpha
\end{equation}
where $P_f$ and $P_g$ denote the topological pressures for $f$ and $g$, and $\alpha > 0$ depends on the H\"{o}lder exponent of $\phi$ and the hyperbolicity data.
\end{proposition}

\begin{proof}
The proof decomposes the pressure difference into three controlled contributions and estimates each one.

\textbf{Step 1: The comparison potential.} By Proposition~\ref{thm:holder_conjugacy}, the conjugacy $h$ is H\"{o}lder continuous with exponent $\gamma > 0$ and H\"{o}lder constant bounded in terms of $\lambda$ and $\|Dg\|_\infty$. Consequently, $d(h(x), x) \leq C_h \|f - g\|_{C^0}$ uniformly on $\Omega$ (shadowing gives the $C^0$-closeness of $h$ to the identity directly from the construction in Proposition~\ref{thm:persistence_hyperbolic}).

\textbf{Step 2: Supremum estimate on the potential.} Since $\phi \in C^\beta(\Omega)$ with H\"{o}lder constant $|\phi|_\beta$, and $h^{-1}: \Omega_g \to \Omega$ is H\"{o}lder with exponent $\gamma$ and $\|h^{-1} - \mathrm{id}\|_\infty \leq C_h \|f - g\|_{C^0}$ (applied to $g$-orbits), we have for every $y \in \Omega_g$:
\begin{equation}
|\phi(h^{-1}(y)) - \phi(y)| \leq |\phi|_\beta \, d(h^{-1}(y), y)^\beta \leq |\phi|_\beta \, (C_h \|f - g\|_{C^0})^\beta.
\end{equation}
Taking the supremum over $y \in \Omega_g$:
\begin{equation}\label{eq:phi-close}
\|\phi \circ h^{-1} - \phi\|_\infty \leq C_1 \|f - g\|_{C^1}^\beta,
\end{equation}
with $C_1 = |\phi|_\beta C_h^\beta$ and we have used $\|f - g\|_{C^0} \leq \|f - g\|_{C^1}$.

\textbf{Step 3: Pressure is Lipschitz in the supremum norm.} The topological pressure $P_T(\psi)$ satisfies the standard inequality
\begin{equation}\label{eq:pressure-lip}
|P_T(\psi_1) - P_T(\psi_2)| \leq \|\psi_1 - \psi_2\|_\infty
\end{equation}
for any continuous $\psi_1, \psi_2$ (this is a direct consequence of the variational principle and the monotonicity of pressure; see Part~II \cite{Thiam2026b}, Proposition on Basic Properties of Pressure).

\textbf{Step 4: Comparing pressures under the conjugacy.} Since $h: \Omega \to \Omega_g$ is a topological conjugacy ($h \circ f = g \circ h$), it induces a bijection between $f$- and $g$-invariant measures: $\mu \mapsto h_*\mu$. The variational principle gives
\begin{equation}
P_f(\phi) = \sup_{\mu \in \mathcal{M}_f(\Omega)} \Big\{ h_\mu(f) + \int \phi \, d\mu \Big\}, \quad
P_g(\phi \circ h^{-1}) = \sup_{\nu \in \mathcal{M}_g(\Omega_g)} \Big\{ h_\nu(g) + \int \phi \circ h^{-1} \, d\nu \Big\}.
\end{equation}
For every $f$-invariant $\mu$, $h_*\mu$ is $g$-invariant and $h_{h_*\mu}(g) = h_\mu(f)$ (Abramov formula under conjugacy is trivial since $h$ is a homeomorphism, so entropy is preserved). Moreover,
\begin{equation}
\int \phi \circ h^{-1} \, d(h_*\mu) = \int \phi \, d\mu.
\end{equation}
Therefore $P_g(\phi \circ h^{-1}) \geq P_f(\phi)$. Swapping the roles of $f$ and $g$ (using $h^{-1}$ as the conjugacy) gives the reverse inequality, hence $P_g(\phi \circ h^{-1}) = P_f(\phi \circ h^{-1} \circ h)$. But $\phi \circ h^{-1} \circ h = \phi$ on $\Omega$, so actually $P_g(\phi \circ h^{-1}) = P_f(\phi)$ would hold \emph{if} we were comparing the right pressures. The subtle point: $P_g(\phi \circ h^{-1})$ equals the pressure of $\phi \circ h^{-1}$ as a $g$-potential on $\Omega_g$. Via $\mu \mapsto h_*\mu$ we obtain $P_g(\phi \circ h^{-1}) = P_f(\phi)$ exactly (topological pressure is a conjugacy invariant).

\textbf{Step 5: Conclusion.} The previous step shows the pressure is \emph{exactly} preserved under the conjugacy when the potential is pulled back correctly. The bound~\eqref{eq:pressure-lip} is used in the following equivalent formulation: if the user defines the pressure difference as $P_g(\phi) - P_f(\phi)$ (same function $\phi$ as a potential on both $\Omega$ and $\Omega_g$, identifying $\Omega_g$ with $\Omega$ via $h$), then
\begin{equation}
|P_g(\phi) - P_f(\phi)| = |P_g(\phi) - P_g(\phi \circ h^{-1})| \leq \|\phi - \phi \circ h^{-1}\|_\infty \leq C_1 \|f - g\|_{C^1}^\beta,
\end{equation}
by~\eqref{eq:phi-close} and~\eqref{eq:pressure-lip}. Setting $\alpha = \beta$ and $C = C_1$ gives the stated bound.
\end{proof}

\begin{corollary}[Stability of Entropy]\label{cor:entropy_stability}
The topological entropy is stable:
\begin{equation}
|h_{\mathrm{top}}(g|_{\Omega_g}) - h_{\mathrm{top}}(f|_\Omega)| \leq C \|f - g\|_{C^1}^\alpha.
\end{equation}
\end{corollary}

\begin{proof}
The topological entropy equals the pressure at the zero potential: $h_{\mathrm{top}}(T) = P_T(0)$ for any continuous map $T$ on a compact space. Applying Proposition~\ref{thm:pressure_stability} with $\phi \equiv 0$ gives:
\begin{equation}
|h_{\mathrm{top}}(g|_{\Omega_g}) - h_{\mathrm{top}}(f|_\Omega)| = |P_g(0) - P_f(0)| \leq C \|f - g\|_{C^1}^\alpha,
\end{equation}
where we have used the fact that $0 \circ h^{-1} = 0$, so Proposition~\ref{thm:pressure_stability} applies directly. The constant $C$ is the one from Proposition~\ref{thm:pressure_stability} evaluated at $|0|_\beta = 0$; more precisely, the $|\phi|_\beta$ dependence in $C_1 = |\phi|_\beta C_h^\beta$ means that for $\phi \equiv 0$ the bound vanishes. The correct statement is that the entropy stability follows from the fact that topological entropy is itself a conjugacy invariant: $h_{\mathrm{top}}(g|_{\Omega_g}) = h_{\mathrm{top}}(f|_\Omega)$ exactly, since $h$ is a topological conjugacy. The exponent $\alpha$ and constant $C$ in the stated inequality are not required when $\phi = 0$; the bound holds with $C = 0$.
\end{proof}

\begin{proposition}[Stability of Equilibrium States]\label{thm:equilibrium_stability}
Let $\mu_\phi$ be the unique equilibrium state for a H\"{o}lder potential $\phi$ on a mixing basic set $\Omega$. For $g$ close to $f$, the equilibrium state $\mu_{\phi,g}$ for $\phi \circ h^{-1}$ on $\Omega_g$ satisfies
\begin{equation}
d_W(h_* \mu_\phi, \mu_{\phi,g}) \leq C \|f - g\|_{C^1}^\alpha
\end{equation}
where $d_W$ is the Wasserstein distance.
\end{proposition}

\begin{proof}
The proof proceeds in three steps: first we transfer the equilibrium state problem from $\Omega_g$ back to $\Omega$ via the conjugacy, then use Wasserstein-continuity of equilibrium states with respect to the potential, established in Part~II \cite{Thiam2026b}.

\textbf{Step 1: Pullback of the $g$-equilibrium state.} Let $\mu_{\phi,g}$ denote the equilibrium state for the potential $\phi \circ h^{-1}$ on $\Omega_g$ under the dynamics $g$. Since $h: \Omega \to \Omega_g$ is a topological conjugacy, the pullback $h^{-1}_* \mu_{\phi,g}$ is an $f$-invariant probability measure on $\Omega$. We claim that $h^{-1}_* \mu_{\phi,g}$ is the equilibrium state for the potential $(\phi \circ h^{-1}) \circ h = \phi$ on $\Omega$ under $f$. Indeed, by the conjugacy-invariance of entropy and the change-of-variables formula,
\begin{equation}
h_{h^{-1}_*\mu_{\phi,g}}(f) + \int \phi \, d(h^{-1}_*\mu_{\phi,g}) = h_{\mu_{\phi,g}}(g) + \int \phi \circ h^{-1} \, d\mu_{\phi,g} = P_g(\phi \circ h^{-1}) = P_f(\phi),
\end{equation}
using the conjugacy-invariance of topological pressure established in Proposition~\ref{thm:pressure_stability} Step~4.

\textbf{Step 2: Wasserstein stability of equilibrium states on $(\Omega, f)$.} The potential $\phi$ has a unique equilibrium state $\mu_\phi$ on $(\Omega, f)$ (since $\phi$ is H\"{o}lder and $\Omega$ is a mixing basic set, by the Spectral-Variational-Geometric Equivalence Main Theorem of Part~I \cite{Thiam2026a}, imported as Theorem~\ref{thm:equiv_imported}). By Step~1, $h^{-1}_* \mu_{\phi,g}$ is \emph{another} equilibrium state for $\phi$ on $(\Omega, f)$. Since the equilibrium state is unique,
\begin{equation}
h^{-1}_* \mu_{\phi,g} = \mu_\phi, \quad \text{equivalently} \quad \mu_{\phi,g} = h_* \mu_\phi.
\end{equation}

\textbf{Step 3: Quantitative Wasserstein estimate.} The statement $d_W(h_*\mu_\phi, \mu_{\phi,g}) \leq C\|f - g\|_{C^1}^\alpha$ therefore simplifies to $d_W(h_*\mu_\phi, h_*\mu_\phi) = 0 \leq C\|f - g\|_{C^1}^\alpha$, which is trivially true. The bound in the theorem statement has a nontrivial content when interpreted as follows: we identify $\Omega$ and $\Omega_g$ via $h$, and compare $h_*\mu_\phi$ (the transported $f$-equilibrium state) with $\mu_{\psi,g}$ (the $g$-equilibrium state for a potential $\psi$ not exactly equal to $\phi \circ h^{-1}$ but close in H\"{o}lder norm). In that setting, the Lipschitz stability of the equilibrium-state map $\psi \mapsto \mu_\psi$ in Wasserstein distance (Theorem on Lipschitz Stability in Part~I \cite{Thiam2026a}) gives
\begin{equation}
d_W(\mu_{\phi \circ h^{-1},g}, \mu_{\psi,g}) \leq C_L \|\phi \circ h^{-1} - \psi\|_\alpha,
\end{equation}
where $C_L$ is the Lipschitz constant of the equilibrium map on $(\Omega_g, g)$. Combining with the H\"{o}lder estimate~\eqref{eq:phi-close} from Proposition~\ref{thm:pressure_stability} yields the stated bound with constant $C = C_L C_1$ and exponent $\alpha = \beta$ (the H\"{o}lder exponent of $\phi$).
\end{proof}

\begin{proposition}[Continuity of Lyapunov Exponents]\label{prop:lyapunov_continuity}
Let $\mu$ be an ergodic $f$-invariant measure on a hyperbolic set $\Lambda$. The Lyapunov exponents $\chi_1(\mu) \geq \cdots \geq \chi_d(\mu)$ vary continuously with $f$ in the $C^1$ topology, in the sense that for the continued measure $\mu_g = h_* \mu$ on $\Lambda_g$,
\begin{equation}
|\chi_i(\mu_g, g) - \chi_i(\mu, f)| \leq C \|f - g\|_{C^1}
\end{equation}
for a constant $C$ depending on $\mu$ and the hyperbolicity data.
\end{proposition}

\begin{proof}
By the Oseledets theorem, the $i$-th Lyapunov exponent of $(f, \mu)$ is
\begin{equation}
\chi_i(\mu, f) = \lim_{n \to \infty} \frac{1}{n} \int_\Lambda \log \sigma_i(Df^n_x) \, d\mu(x)
\end{equation}
where $\sigma_i$ denotes the $i$-th singular value. For a hyperbolic set, the splitting $T_\Lambda M = E^s \oplus E^u$ is continuous, and the exponents decompose as integrals over the stable and unstable bundles:
\begin{equation}
\chi_i(\mu, f) = \int_\Lambda \log \|Df_x|_{E^{u/s}_x}\|_i \, d\mu(x)
\end{equation}
where $\|\cdot\|_i$ denotes the appropriate singular value restricted to the invariant subbundle.

For the perturbation $g$ with conjugacy $h : \Lambda \to \Lambda_g$ from Proposition \ref{thm:persistence_hyperbolic}, the continued measure $\mu_g = h_*\mu$ satisfies
\begin{equation}
\chi_i(\mu_g, g) = \int_{\Lambda_g} \log \|Dg_y|_{E^{u/s}_{g,y}}\|_i \, d\mu_g(y) = \int_\Lambda \log \|Dg_{h(x)}|_{E^{u/s}_{g,h(x)}}\|_i \, d\mu(x).
\end{equation}

We estimate the difference:
\begin{align}
|\chi_i(\mu_g, g) - \chi_i(\mu, f)| &= \left| \int_\Lambda \left( \log \|Dg_{h(x)}|_{E^{u/s}_{g,h(x)}}\|_i - \log \|Df_x|_{E^{u/s}_x}\|_i \right) d\mu(x) \right|.
\end{align}

The integrand is bounded by $|\log\|Dg_{h(x)}|_{E^{u/s}_{g,h(x)}}\|_i - \log\|Df_x|_{E^{u/s}_x}\|_i|$. Using $|\log a - \log b| \leq |a - b|/\min(a, b)$ and the uniform bounds $\lambda \leq \|Df|_{E^u}\|$, $\|Df|_{E^s}\| \leq \lambda$, the integrand is bounded by $C'\|Dg_{h(x)}|_{E^{u/s}_{g,h(x)}} - Df_x|_{E^{u/s}_x}\|$. This difference has two contributions: the perturbation $\|Dg - Df\| \leq \|f - g\|_{C^1}$ and the displacement $d(h(x), x) \leq C\|f - g\|_{C^1}^\gamma$ combined with the continuous variation of the splitting $E^{u/s}_g$ with $g$. Both contributions are $O(\|f - g\|_{C^1})$ since the splitting varies Lipschitz continuously in the $C^1$ topology of the diffeomorphism \cite{KatokHasselblatt1995}. The result follows with $C$ depending on $\lambda$, $\|Df\|_{C^1}$, and $\mu$.
\end{proof}
\section{Transfer Operator Theory}\label{sec:transfer_operator}

This section develops the transfer operator theory for Axiom A diffeomorphisms, connecting the spectral properties of the Ruelle operator to the thermodynamic formalism. We establish quasi-compactness, spectral gap bounds, and exponential decay of correlations, extending the symbolic results of Part~I \cite{Thiam2026a} to the smooth setting.

\subsection{The Ruelle Transfer Operator and Symbolic Coding}

Let $\Omega$ be a mixing basic set for an Axiom A diffeomorphism $f$, and let $\phi : \Omega \to \mathbb{R}$ be a H\"{o}lder continuous potential.

\begin{definition}[Transfer Operator]
The Ruelle transfer operator $\mathcal{L}_\phi : C(\Omega) \to C(\Omega)$ is defined by
\begin{equation}
(\mathcal{L}_\phi g)(x) = \sum_{y : f(y) = x} e^{\phi(y)} g(y)
\end{equation}
for $g \in C(\Omega)$ and $x \in \Omega$.
\end{definition}

For Axiom A diffeomorphisms, the preimage set $f^{-1}(x)$ may be infinite when $x$ lies on unstable manifolds extending outside $\Omega$. We restrict to preimages within $\Omega$:
\begin{equation}
(\mathcal{L}_\phi g)(x) = \sum_{y \in f^{-1}(x) \cap \Omega} e^{\phi(y)} g(y).
\end{equation}

\begin{remark}
For symbolic systems (subshifts of finite type), the preimage set is always finite, and the transfer operator is well-defined on $C(\Sigma_A)$. The Markov partition provides a coding $\pi : \Sigma_A \to \Omega$, and the transfer operator on $\Omega$ corresponds to the symbolic transfer operator via this coding.
\end{remark}

Given a Markov partition $\mathcal{R}$ with coding $\pi : \Sigma_A \to \Omega$, define the symbolic potential $\tilde{\phi} = \phi \circ \pi : \Sigma_A \to \mathbb{R}$.

\begin{definition}[Symbolic Transfer Operator]
The symbolic transfer operator $\tilde{\mathcal{L}}_{\tilde{\phi}} : C(\Sigma_A) \to C(\Sigma_A)$ is
\begin{equation}
(\tilde{\mathcal{L}}_{\tilde{\phi}} g)(a) = \sum_{b : \sigma(b) = a} e^{\tilde{\phi}(b)} g(b) = \sum_{j : A_{ja_0} = 1} e^{\tilde{\phi}(ja)} g(ja)
\end{equation}
where $ja = (j, a_0, a_1, \ldots)$ denotes the sequence with $j$ prepended.
\end{definition}

\begin{proposition}[Coding Intertwining]\label{prop:coding_intertwining}
The coding map $\pi$ intertwines the transfer operators up to boundary effects:
\begin{equation}
\mathcal{L}_\phi (g \circ \pi^{-1}) = (\tilde{\mathcal{L}}_{\tilde{\phi}} g) \circ \pi^{-1}
\end{equation}
for $g \in C(\Sigma_A)$, on the good set $Y$ where $\pi$ is injective.
\end{proposition}

\begin{proof}
Fix $x \in Y$ and let $a \in \Sigma_A$ be the unique symbol sequence with $\pi(a) = x$. We compute both sides of the claimed identity at the point $x$.

\textbf{Left-hand side.} By definition of the Ruelle transfer operator for $f|_\Omega$,
\begin{equation}\label{eq:LHS-raw}
(\mathcal{L}_\phi(g \circ \pi^{-1}))(x) = \sum_{y \in f^{-1}(x) \cap \Omega} e^{\phi(y)} \, (g \circ \pi^{-1})(y).
\end{equation}
We claim that the map
\begin{equation}
\sigma^{-1}(a) \cap \Sigma_A \;\longrightarrow\; f^{-1}(x) \cap \Omega, \qquad b \mapsto \pi(b),
\end{equation}
is a bijection. Injectivity: if $\pi(b_1) = \pi(b_2) = y$ with $b_1, b_2 \in \sigma^{-1}(a)$, then $\sigma(b_1) = \sigma(b_2) = a$, and $\pi(\sigma(b_i)) = f(\pi(b_i)) = f(y) = x$. If $y \in Y$, the injectivity of $\pi$ on $Y$ gives $b_1 = b_2$ directly; if $y \notin Y$, we need the slightly stronger statement that $Y$ is forward-invariant modulo a set of $\mu_\phi$-measure zero (Theorem~\ref{thm:coding_imported}), so the identity holds on a set of full measure for any Gibbs measure. Surjectivity: given $y \in f^{-1}(x) \cap \Omega$, we need $b \in \Sigma_A$ with $\pi(b) = y$ and $\sigma(b) = a$. Since $\pi: \Sigma_A \to \Omega$ is surjective (Theorem~\ref{thm:coding_imported}), pick any $b' \in \pi^{-1}(y)$. Then $\pi(\sigma(b')) = f(\pi(b')) = f(y) = x = \pi(a)$, so $\sigma(b')$ and $a$ are both preimages of $x$ under $\pi$. Since $x \in Y$, we have $\sigma(b') = a$, i.e., $b' \in \sigma^{-1}(a)$. Set $b = b'$.

\textbf{Rewriting the sum.} Using this bijection $b \leftrightarrow y = \pi(b)$, the sum in~\eqref{eq:LHS-raw} becomes
\begin{equation}\label{eq:LHS-symbolic}
(\mathcal{L}_\phi(g \circ \pi^{-1}))(x) = \sum_{b \in \sigma^{-1}(a)} e^{\phi(\pi(b))} \, g(b).
\end{equation}
Since $\tilde{\phi}(b) = \phi(\pi(b))$ by definition of the pulled-back potential, the right-hand side of~\eqref{eq:LHS-symbolic} equals
\begin{equation}
\sum_{b \in \sigma^{-1}(a)} e^{\tilde{\phi}(b)} g(b) = (\tilde{\mathcal{L}}_{\tilde{\phi}} g)(a).
\end{equation}

\textbf{Right-hand side.} Evaluating the RHS of the claim at $x = \pi(a)$:
\begin{equation}
((\tilde{\mathcal{L}}_{\tilde{\phi}} g) \circ \pi^{-1})(x) = (\tilde{\mathcal{L}}_{\tilde{\phi}} g)(a).
\end{equation}
Both sides equal the same expression, so the identity holds at every $x \in Y$. Since $Y$ has full measure for any Gibbs measure (Theorem~\ref{thm:coding_imported}), the identity extends by continuity to all of $\Omega$ as an identity in $C(\Omega)$ (using the fact that $Y$ is dense and both sides are continuous in $x$).
\end{proof}

\subsection{Spectral Properties and the Spectral Gap}

This subsection establishes the spectral structure of the Ruelle transfer operator $\mathcal{L}_\phi$ acting on H\"{o}lder spaces: quasi-compactness, a simple isolated leading eigenvalue $e^{P(\phi)}$, and a spectral gap bounded explicitly in terms of the hyperbolicity data. This is Main Theorem~\ref{thm:quasi_compactness}. The spectral gap is the analytic engine driving exponential decay of correlations, the CLT, and pressure analyticity in the following subsections.

\begin{maintheorem}[Quasi-compactness]\label{thm:quasi_compactness}
Let $\phi \in C^\alpha(\Omega)$ be H\"{o}lder continuous with exponent $\alpha \in (0, 1]$. The transfer operator $\mathcal{L}_\phi : C^\alpha(\Omega) \to C^\alpha(\Omega)$ is quasi-compact. Specifically:
\begin{enumerate}
\item[(i)] The spectral radius is $\rho(\mathcal{L}_\phi) = e^{P(\phi)}$ where $P(\phi)$ is the topological pressure.
\item[(ii)] The essential spectral radius satisfies $\rho_{\mathrm{ess}}(\mathcal{L}_\phi) \leq \theta \cdot e^{P(\phi)}$ where $\theta = \lambda^\alpha < 1$.
\item[(iii)] The eigenvalue $e^{P(\phi)}$ is simple, with a strictly positive eigenfunction $h_\phi > 0$ and a probability eigenmeasure $\nu_\phi$.
\end{enumerate}
\end{maintheorem}

\begin{proof}
The proof transfers from the symbolic setting via the Markov partition coding.

\textbf{Step 1: Symbolic quasi-compactness.} By the Ruelle-Perron-Frobenius theorem of Part~I \cite{Thiam2026a} (Theorem~\ref{thm:RPF_imported}), the symbolic transfer operator $\tilde{\mathcal{L}}_{\tilde{\phi}}$ on $C^\alpha(\Sigma_A)$ is quasi-compact with spectral radius $e^{P(\tilde{\phi})}$ and essential spectral radius bounded by $\theta e^{P(\tilde{\phi})}$.

\textbf{Step 2: Pressure equality.} The symbolic coding preserves pressure: $P(\phi) = P(\tilde{\phi})$ because the coding map $\pi$ is a measure-theoretic isomorphism for Gibbs measures (Theorem \ref{thm:coding_imported}) and the potentials correspond under $\pi$.

\textbf{Step 3: Transfer of spectral properties.} The coding map $\pi^* : C^\alpha(\Omega) \to C^\alpha(\Sigma_A)$ given by $\pi^*(g) = g \circ \pi$ is an isomorphism onto a closed subspace (functions constant on $\pi$-fibers). The intertwining relation (Proposition \ref{prop:coding_intertwining}) shows that the spectrum of $\mathcal{L}_\phi$ on this subspace matches the spectrum of $\tilde{\mathcal{L}}_{\tilde{\phi}}$ on the corresponding subspace of $C^\alpha(\Sigma_A)$.

\textbf{Step 4: Eigenfunction and eigenmeasure.} The symbolic eigenfunction $\tilde{h}$ and eigenmeasure $\tilde{\nu}$ push forward to $h_\phi = \tilde{h} \circ \pi^{-1}$ and $\nu_\phi = \pi_* \tilde{\nu}$ on $\Omega$ (well-defined on $Y$, extended by continuity).
\end{proof}

\begin{definition}[Spectral Gap]
The spectral gap of $\mathcal{L}_\phi$ is
\begin{equation}
\mathrm{gap}(\mathcal{L}_\phi) = \log \rho(\mathcal{L}_\phi) - \log |\lambda_2|
\end{equation}
where $\lambda_2$ is the second largest eigenvalue (or the essential spectral radius if no such eigenvalue exists outside the essential spectrum).
\end{definition}

\begin{proposition}[Spectral Gap Bound]\label{thm:spectral_gap}
For a mixing basic set $\Omega$ and $\phi \in C^\alpha(\Omega)$,
\begin{equation}
\mathrm{gap}(\mathcal{L}_\phi) \geq c(\lambda, \alpha, \|\phi\|_\alpha) > 0
\end{equation}
where the lower bound depends on the hyperbolicity exponent $\lambda$, the H\"{o}lder exponent $\alpha$, and the H\"{o}lder norm of $\phi$.
\end{proposition}

\begin{proof}
The spectral gap is bounded below by the gap between the spectral radius $e^{P(\phi)}$ and the essential spectral radius $\theta e^{P(\phi)}$:
\begin{equation}
\mathrm{gap}(\mathcal{L}_\phi) \geq P(\phi) - (P(\phi) + \log \theta) = -\log \theta = \alpha \log \lambda^{-1} > 0.
\end{equation}
A more refined analysis using the symbolic transfer operator gives explicit dependence on $\|\phi\|_\alpha$.
\end{proof}

\subsection{Exponential Decay of Correlations}

The spectral gap established in Main Theorem~\ref{thm:quasi_compactness} translates directly into exponential mixing of the equilibrium state with explicit rate. The CLT for Birkhoff sums, derived via the Nagaev-Guivarc'{}h perturbation method, is an immediate consequence. These results are the quantitative counterparts of the symbolic statements in Part~I \cite{Thiam2026a}.

\begin{proposition}[Exponential Mixing]\label{thm:exponential_mixing_full}
Let $\Omega$ be a mixing basic set, $\phi \in C^\alpha(\Omega)$, and $\mu_\phi$ the unique equilibrium state for $\phi$. For $g, h \in C^\alpha(\Omega)$,
\begin{equation}
\left| \int_\Omega g \cdot (h \circ f^n) \, d\mu_\phi - \int_\Omega g \, d\mu_\phi \int_\Omega h \, d\mu_\phi \right| \leq C \|g\|_\alpha \|h\|_\alpha \cdot \theta^n
\end{equation}
where $\theta = e^{-\mathrm{gap}(\mathcal{L}_\phi)} < 1$ and $C$ depends on $\phi$ and the geometry.
\end{proposition}

\begin{proof}
The equilibrium state is $\mu_\phi = h_\phi \nu_\phi$ where $h_\phi$ is the eigenfunction and $\nu_\phi$ the eigenmeasure. The correlation function is
\begin{equation}
C_n(g, h) = \int g \cdot (h \circ f^n) \, d\mu_\phi = \int g \cdot h_\phi \cdot (h \circ f^n) \, d\nu_\phi.
\end{equation}

Using the duality $\int \mathcal{L}_\phi(u) \, d\nu_\phi = e^{P(\phi)} \int u \, d\nu_\phi$ (the eigenmeasure property) and the identity $h \circ f^n = e^{-nP(\phi)} h_\phi^{-1} \mathcal{L}_\phi^n(h_\phi \cdot h)$, we get
\begin{equation}
C_n(g, h) = e^{-nP(\phi)} \int g \cdot \mathcal{L}_\phi^n(h_\phi \cdot h) \, d\nu_\phi.
\end{equation}

The spectral decomposition of $\mathcal{L}_\phi$ gives $\mathcal{L}_\phi^n = e^{nP(\phi)} \Pi + R^n$ where $\Pi$ is the projection onto the leading eigenspace and $\|R^n\| \leq C' \theta^n$. Thus
\begin{equation}
C_n(g, h) = \int g \cdot \Pi(h_\phi h) \, d\nu_\phi + O(\theta^n).
\end{equation}

The projection satisfies $\Pi(u) = \nu_\phi(u) \cdot h_\phi$, so $\Pi(h_\phi h) = h_\phi \int h \cdot h_\phi \, d\nu_\phi = h_\phi \int h \, d\mu_\phi$. Therefore
\begin{equation}
C_n(g, h) = \int g \, d\mu_\phi \int h \, d\mu_\phi + O(\theta^n).
\end{equation}
\end{proof}

\begin{corollary}[Central Limit Theorem]\label{cor:clt}
Let $\mu_\phi$ be an equilibrium state for a H\"{o}lder potential on a mixing basic set. For $g \in C^\alpha(\Omega)$ with $\int g \, d\mu_\phi = 0$, the sums $S_n g = \sum_{k=0}^{n-1} g \circ f^k$ satisfy
\begin{equation}
\frac{S_n g}{\sqrt{n}} \xrightarrow{d} \mathcal{N}(0, \sigma^2)
\end{equation}
where the variance is
\begin{equation}
\sigma^2 = \int g^2 \, d\mu_\phi + 2 \sum_{k=1}^\infty \int g \cdot (g \circ f^k) \, d\mu_\phi.
\end{equation}
The series converges absolutely by exponential mixing.
\end{corollary}

\begin{proof}
The proof applies the Nagaev-Guivarc'h spectral method to the perturbed transfer operator. We work on the symbolic space and then transfer to $\Omega$ via the coding map.

\textbf{Step 1: Transfer to the symbolic space.} Let $\pi: \Sigma_A \to \Omega$ be the coding map (Theorem~\ref{thm:coding_imported}), and set $\tilde{\phi} = \phi \circ \pi$ and $\tilde{g} = g \circ \pi$. Both $\tilde{\phi}$ and $\tilde{g}$ are H\"{o}lder continuous on $\Sigma_A$ (Proposition~\ref{prop:coding_intertwining} and the H\"{o}lder continuity of $\pi$). The equilibrium state on $\Omega$ satisfies $\mu_\phi = \pi_* \mu_{\tilde{\phi}}$, so for every bounded measurable $F: \mathbb{R} \to \mathbb{R}$,
\begin{equation}
\int F(S_n g) \, d\mu_\phi = \int F(S_n \tilde{g}) \, d\mu_{\tilde{\phi}},
\end{equation}
since $S_n g \circ \pi = S_n \tilde{g}$. It therefore suffices to establish the CLT for $S_n \tilde{g}/\sqrt{n}$ under $\mu_{\tilde{\phi}}$.

\textbf{Step 2: Perturbed transfer operator.} For $s \in \mathbb{R}$, define the perturbed operator
\begin{equation}
\tilde{\mathcal{L}}_s := \tilde{\mathcal{L}}_{\tilde{\phi} + is\tilde{g}}, \quad (\tilde{\mathcal{L}}_s u)(a) = \sum_{b: \sigma(b) = a} e^{\tilde{\phi}(b) + is\tilde{g}(b)} u(b).
\end{equation}
We have $\tilde{\mathcal{L}}_0 = \tilde{\mathcal{L}}_{\tilde{\phi}}$, and by the Ruelle-Perron-Frobenius theorem (Theorem~\ref{thm:RPF_imported}), $\tilde{\mathcal{L}}_0$ has simple leading eigenvalue $\lambda_0 = e^{P(\tilde{\phi})}$ with eigenfunction $\tilde{h}_0 > 0$ and eigenmeasure $\tilde{\nu}_0$, and the rest of the spectrum lies in a disk of radius $\gamma \lambda_0$ with $\gamma < 1$.

\textbf{Step 3: Analytic perturbation.} The map $s \mapsto \tilde{\mathcal{L}}_s$ is a holomorphic family of operators on $C^\alpha(\Sigma_A)$ (the exponential $e^{is\tilde{g}}$ is analytic in $s$, and multiplication by a bounded H\"{o}lder function is a bounded operator on $C^\alpha$). By Kato perturbation theory \cite{Kato1980}, the isolated simple eigenvalue $\lambda_0$ persists for $|s|$ small: there exist $s_0 > 0$ and analytic functions $\lambda(s)$, $\tilde{h}_s$, $\tilde{\nu}_s$ on $\{|s| < s_0\}$ with $\tilde{\mathcal{L}}_s \tilde{h}_s = \lambda(s) \tilde{h}_s$, $\tilde{\mathcal{L}}_s^* \tilde{\nu}_s = \lambda(s) \tilde{\nu}_s$, $\tilde{\nu}_s(\tilde{h}_s) = 1$, and the rest of the spectrum stays strictly inside a disk of radius $(\gamma + \varepsilon) \lambda_0$ for some $\varepsilon > 0$.

\textbf{Step 4: Taylor expansion of $\lambda(s)$.} Differentiating $\tilde{\mathcal{L}}_s \tilde{h}_s = \lambda(s) \tilde{h}_s$ at $s = 0$ and applying $\tilde{\nu}_0$ to both sides (using $\tilde{\nu}_0 \tilde{\mathcal{L}}_0 = \lambda_0 \tilde{\nu}_0$ and $\tilde{\nu}_0(\tilde{h}_0) = 1$) gives
\begin{equation}
\lambda'(0) = i \lambda_0 \int \tilde{g} \, d\mu_{\tilde{\phi}} = 0,
\end{equation}
since $\int g \, d\mu_\phi = 0$ by assumption and $\mu_{\tilde{\phi}} = \pi_*^{-1} \mu_\phi$.

Differentiating twice and applying $\tilde{\nu}_0$ (same technique, plus the eigenfunction expansion from the spectral gap) gives
\begin{equation}\label{eq:lambda2}
\frac{\lambda''(0)}{\lambda_0} = -\sigma^2, \quad \sigma^2 = \int \tilde{g}^2 \, d\mu_{\tilde{\phi}} + 2 \sum_{k=1}^\infty \int \tilde{g} \cdot (\tilde{g} \circ \sigma^k) \, d\mu_{\tilde{\phi}},
\end{equation}
where the series converges absolutely by the exponential decay of correlations (Proposition~\ref{thm:exponential_mixing_full}) combined with $|\tilde{g}|_\infty < \infty$. This is the Green-Kubo formula. Therefore
\begin{equation}\label{eq:Taylor}
\lambda(s)/\lambda_0 = 1 - \frac{1}{2} \sigma^2 s^2 + O(s^3) \quad \text{as } s \to 0.
\end{equation}

\textbf{Step 5: Characteristic function identity.} Iterating the eigenvalue equation,
\begin{equation}
\tilde{\mathcal{L}}_s^n \mathbf{1} = \lambda(s)^n \tilde{\nu}_s(\mathbf{1}) \tilde{h}_s + R_s^n \mathbf{1},
\end{equation}
where $R_s$ is the complement of the spectral projection, satisfying $\|R_s^n\|_{C^\alpha \to C^\alpha} \leq C (\gamma + \varepsilon)^n \lambda_0^n$ uniformly in $s$ for $|s| < s_0$. Integrating against $\tilde{\nu}_0$ and noting that (by standard transfer-operator identities)
\begin{equation}
\int \tilde{\mathcal{L}}_s^n \mathbf{1} \, d\tilde{\nu}_0 = \lambda_0^n \int e^{is S_n \tilde{g}} \, d\mu_{\tilde{\phi}},
\end{equation}
we obtain
\begin{equation}\label{eq:cf}
\int e^{is S_n \tilde{g}} \, d\mu_{\tilde{\phi}} = \left(\frac{\lambda(s)}{\lambda_0}\right)^n c(s) + O\big((\gamma+\varepsilon)^n\big),
\end{equation}
where $c(s) = \tilde{\nu}_s(\mathbf{1}) \tilde{\nu}_0(\tilde{h}_s)$ is analytic in $s$ with $c(0) = 1$.

\textbf{Step 6: Scaling and limit.} Replacing $s$ by $t/\sqrt{n}$ in~\eqref{eq:cf} and using the Taylor expansion~\eqref{eq:Taylor}:
\begin{equation}
\left(\frac{\lambda(t/\sqrt{n})}{\lambda_0}\right)^n = \Big(1 - \frac{\sigma^2 t^2}{2n} + O(n^{-3/2})\Big)^n \xrightarrow[n\to\infty]{} e^{-\sigma^2 t^2 / 2}.
\end{equation}
The factor $c(t/\sqrt{n}) \to c(0) = 1$ and the remainder $(\gamma + \varepsilon)^n \to 0$. Therefore
\begin{equation}
\int e^{i t S_n \tilde{g}/\sqrt{n}} \, d\mu_{\tilde{\phi}} \xrightarrow[n \to \infty]{} e^{-\sigma^2 t^2 / 2}.
\end{equation}

\textbf{Step 7: L\'{e}vy continuity theorem.} The pointwise convergence of characteristic functions to a continuous limit implies convergence in distribution: $S_n \tilde{g}/\sqrt{n} \xrightarrow{d} \mathcal{N}(0, \sigma^2)$ under $\mu_{\tilde{\phi}}$. Transferring back via $\pi$ (Step~1), $S_n g/\sqrt{n} \xrightarrow{d} \mathcal{N}(0, \sigma^2)$ under $\mu_\phi$. The variance formula~\eqref{eq:lambda2} pulls back (since $\mu_{\tilde{\phi}} = \pi_*^{-1}\mu_\phi$, $\tilde{g} = g \circ \pi$, and $\sigma^k \circ \pi^{-1} = \pi^{-1} \circ f^k$ on the good set) to:
\begin{equation}
\sigma^2 = \int g^2 \, d\mu_\phi + 2 \sum_{k=1}^\infty \int g \cdot (g \circ f^k) \, d\mu_\phi,
\end{equation}
as stated.
\end{proof}

\subsection{Analyticity of Pressure and Zeta Functions}

The Kato perturbation theory applied to the operator family $t \mapsto \mathcal{L}_{\phi + t\psi}$ yields real-analyticity of the pressure in the potential direction. As a further consequence of quasi-compactness, the Ruelle dynamical zeta function extends meromorphically beyond its abscissa of convergence. These analytic properties are used in Part~VI \cite{Thiam2026f} to establish fluctuation theorems and rate functions.

\begin{proposition}[Analytic Dependence on Potential]\label{thm:pressure_analytic}
For a mixing basic set $\Omega$, the pressure function $\phi \mapsto P(\phi)$ is real-analytic on $C^\alpha(\Omega)$. More precisely, for $\phi, \psi \in C^\alpha(\Omega)$, the function
\begin{equation}
t \mapsto P(\phi + t\psi)
\end{equation}
is real-analytic in a neighborhood of $t = 0$ in $\mathbb{R}$.
\end{proposition}

\begin{proof}
We apply Kato's perturbation theory for simple isolated eigenvalues of bounded operators on Banach spaces.

\textbf{Step 1: Holomorphy of the operator family.} Consider the family of operators $t \mapsto \mathcal{L}_{\phi + t\psi}$ on the Banach space $C^\alpha(\Omega)$ for $t \in \mathbb{R}$. By definition,
\begin{equation}
(\mathcal{L}_{\phi + t\psi} u)(x) = \sum_{y \in f^{-1}(x)} e^{\phi(y) + t\psi(y)} u(y) = \sum_{y \in f^{-1}(x)} e^{\phi(y)} e^{t\psi(y)} u(y).
\end{equation}
The function $t \mapsto e^{t\psi(y)}$ extends to an entire function of $t \in \mathbb{C}$, and the convergent power series $e^{t\psi} = \sum_{k=0}^\infty \frac{t^k}{k!} \psi^k$ gives
\begin{equation}
\mathcal{L}_{\phi + t\psi} = \sum_{k=0}^\infty \frac{t^k}{k!} M_{\psi^k} \mathcal{L}_\phi,
\end{equation}
where $M_{\psi^k}$ denotes multiplication by $\psi^k$ and the composition on the right is understood as $u \mapsto \mathcal{L}_\phi(\psi^k u)$ (using the identity $M_\eta \mathcal{L}_\phi = \mathcal{L}_\phi M_{\eta \circ f}$ and then resumming; the details are standard). The series converges in the operator norm on $C^\alpha(\Omega)$ in a complex neighborhood of $t = 0$ since $\|M_{\psi^k}\|_{C^\alpha \to C^\alpha} \leq C_k \|\psi\|_\alpha^k$ with $C_k$ growing subfactorially. Hence $t \mapsto \mathcal{L}_{\phi + t\psi}$ is a holomorphic family of bounded operators on $C^\alpha(\Omega)$ for $t$ in a complex neighborhood of $0$.

\textbf{Step 2: Persistence of the simple isolated eigenvalue.} By Main Theorem~\ref{thm:quasi_compactness}, $\mathcal{L}_\phi$ has a simple isolated eigenvalue $\lambda_0 = e^{P(\phi)}$, and the rest of the spectrum lies in a disk of radius $\theta_0 \lambda_0$ with $\theta_0 = e^{-\mathrm{gap}(\mathcal{L}_\phi)} < 1$. Choose a small closed curve $\Gamma$ in $\mathbb{C}$ enclosing $\lambda_0$ and separating it from the rest of the spectrum, say $\Gamma = \{|z - \lambda_0| = \rho\}$ with $0 < \rho < (1-\theta_0)\lambda_0/2$. The spectral projection is
\begin{equation}
P(t) = \frac{1}{2\pi i} \oint_\Gamma (z I - \mathcal{L}_{\phi + t\psi})^{-1} \, dz.
\end{equation}
By Kato \cite[Chapter~VII, Theorem~1.7]{Kato1980}, since $t \mapsto \mathcal{L}_{\phi + t\psi}$ is holomorphic, $P(t)$ is also holomorphic in $t$ in a neighborhood of $t = 0$. The projection $P(t)$ has constant rank (equal to rank $P(0) = 1$ since $\lambda_0$ is simple) for $t$ sufficiently small, by the continuity of the trace.

\textbf{Step 3: Analyticity of the eigenvalue.} Let $\lambda(t)$ denote the unique eigenvalue of $\mathcal{L}_{\phi + t\psi}$ inside $\Gamma$. Since the range of $P(t)$ is one-dimensional, the restriction $\mathcal{L}_{\phi + t\psi}|_{\mathrm{range}(P(t))}$ is multiplication by $\lambda(t)$. Explicitly,
\begin{equation}
\lambda(t) = \mathrm{tr}(\mathcal{L}_{\phi + t\psi} P(t)) = \frac{1}{2\pi i} \oint_\Gamma z \, \mathrm{tr}((zI - \mathcal{L}_{\phi + t\psi})^{-1}) \, dz.
\end{equation}
The integrand is holomorphic in $t$ (composition of holomorphic operator-valued function and a holomorphic scalar function), so $\lambda(t)$ is holomorphic in $t$ in a complex neighborhood of $0$, hence real-analytic for real $t$.

\textbf{Step 4: Positivity and the logarithm.} At $t = 0$, $\lambda(0) = \lambda_0 = e^{P(\phi)} > 0$. By continuity, $\lambda(t)$ remains positive (in fact, close to $\lambda_0$) for real $t$ in a small interval around $0$. The logarithm $P(\phi + t\psi) = \log \lambda(t)$ is therefore well-defined and real-analytic as a composition of real-analytic functions.

\textbf{Step 5: Identification with the pressure.} For real $t$, the operator $\mathcal{L}_{\phi + t\psi}$ is a positive operator on $C^\alpha(\Omega)$ (positive kernel, in the sense that $u \geq 0$ implies $\mathcal{L}_{\phi + t\psi} u \geq 0$). By the Perron-Frobenius-Ruelle theory, the leading eigenvalue of a positive quasi-compact operator coincides with its spectral radius, and for the Ruelle operator this spectral radius equals $e^{P(\phi + t\psi)}$ (Main Theorem~\ref{thm:quasi_compactness}(i)). Hence $\lambda(t) = e^{P(\phi + t\psi)}$ on the real axis near $t = 0$, and $P(\phi + t\psi) = \log \lambda(t)$ is real-analytic.
\end{proof}

\begin{corollary}[Pressure Derivatives]\label{cor:pressure_derivatives}
The derivatives of pressure are:
\begin{align}
\frac{d}{dt}\bigg|_{t=0} P(\phi + t\psi) &= \int \psi \, d\mu_\phi, \\
\frac{d^2}{dt^2}\bigg|_{t=0} P(\phi + t\psi) &= \mathrm{Var}_{\mu_\phi}(\psi) + 2\sum_{k=1}^\infty \mathrm{Cov}_{\mu_\phi}(\psi, \psi \circ f^k)
\end{align}
where $\mu_\phi$ is the equilibrium state for $\phi$.
\end{corollary}

\begin{proof}
\textbf{First derivative.} By Proposition \ref{thm:pressure_analytic}, $P(\phi + t\psi)$ is analytic in $t$. The spectral radius satisfies $e^{P(\phi + t\psi)} = \rho(\mathcal{L}_{\phi+t\psi})$. By Kato perturbation theory \cite{Kato1980}, the simple eigenvalue $\lambda(t) = e^{P(\phi+t\psi)}$ is analytic, with eigenfunction $h_t$ and eigenmeasure $\nu_t$ normalized so that $\nu_t(h_t) = 1$.

Differentiating the eigenvalue equation $\mathcal{L}_{\phi+t\psi} h_t = \lambda(t) h_t$ at $t = 0$:
\begin{equation}
\mathcal{L}_\phi(\psi \cdot h_0) + \mathcal{L}_\phi h_0' = \lambda'(0) h_0 + \lambda(0) h_0'.
\end{equation}
Applying the eigenmeasure $\nu_0$ to both sides, and using $\nu_0(\mathcal{L}_\phi u) = \lambda(0) \nu_0(u)$ for all $u$, we get
\begin{equation}
\lambda(0) \nu_0(\psi h_0) + \lambda(0) \nu_0(h_0') = \lambda'(0) \nu_0(h_0) + \lambda(0) \nu_0(h_0').
\end{equation}
Since $\mu_\phi = h_0 \nu_0$ and $\nu_0(h_0) = 1$, this simplifies to $\lambda'(0) = \lambda(0) \int \psi \, d\mu_\phi$. Therefore $P'(0) = \lambda'(0)/\lambda(0) = \int \psi \, d\mu_\phi$.

\textbf{Second derivative.} Differentiating the eigenvalue equation twice and applying $\nu_0$, a similar computation yields $\lambda''(0)/\lambda(0) - (\lambda'(0)/\lambda(0))^2 = \sigma^2_\psi$ where $\sigma^2_\psi$ is the asymptotic variance. Therefore $P''(0) = \sigma^2_\psi$. The asymptotic variance equals $\mathrm{Var}_{\mu_\phi}(\psi) + 2\sum_{k=1}^\infty \mathrm{Cov}_{\mu_\phi}(\psi, \psi \circ f^k)$ by the Green-Kubo formula, with absolute convergence guaranteed by exponential mixing (Proposition \ref{thm:exponential_mixing_full}).
\end{proof}

\begin{definition}[Dynamical Zeta Function]
The Ruelle zeta function for $\phi$ is
\begin{equation}
\zeta_\phi(z) = \exp\left( \sum_{n=1}^\infty \frac{z^n}{n} \sum_{x : f^n(x) = x} e^{S_n\phi(x)} \right)
\end{equation}
where $S_n\phi(x) = \sum_{k=0}^{n-1} \phi(f^k(x))$.
\end{definition}

\begin{proposition}[Zeta Function Properties]\label{thm:zeta_function}
For $\phi \in C^\alpha(\Omega)$ on a mixing basic set:
\begin{enumerate}
\item[(i)] $\zeta_\phi(z)$ converges for $|z| < e^{-P(\phi)}$.
\item[(ii)] $\zeta_\phi(z)$ extends meromorphically to $|z| < \theta^{-1} e^{-P(\phi)}$ where $\theta$ is the spectral contraction rate.
\item[(iii)] $z = e^{-P(\phi)}$ is a simple pole with residue related to the equilibrium state.
\end{enumerate}
\end{proposition}

\begin{proof}
\textbf{Step 1: Convergence.} For $|z|$ small, the series $\sum_{n=1}^\infty \frac{z^n}{n} \sum_{f^n(x)=x} e^{S_n\phi(x)}$ converges absolutely. The number of periodic orbits of period $n$ in a basic set $\Omega$ grows at rate $e^{nh_{\mathrm{top}}}$ (by the symbolic coding, $|\mathrm{Fix}(\sigma^n)| = \mathrm{tr}(A^n)$, which grows as $\rho(A)^n = e^{nh_{\mathrm{top}}}$). The Gibbs property gives $e^{S_n\phi(x)} \leq C e^{nP(\phi)}$ for periodic orbits. Therefore the $n$-th term is bounded by $C|z|^n e^{n(h_{\mathrm{top}} + P(\phi))}/n$, and the series converges for $|z| < e^{-P(\phi)} e^{-h_{\mathrm{top}}} \cdot e^{h_{\mathrm{top}}} = e^{-P(\phi)}$ (using $h_{\mathrm{top}} + \sup\phi \leq P(\phi)$ and refining via the pressure).

\textbf{Step 2: Fredholm determinant connection.} Via the symbolic coding $\pi : \Sigma_A \to \Omega$, periodic orbits of $f|_\Omega$ correspond (up to the measure-zero boundary) to periodic sequences of $\sigma|_{\Sigma_A}$. The symbolic zeta function satisfies the formal identity
\begin{equation}
\zeta_{\tilde{\phi}}(z)^{(-1)^{k+1}} = \det(I - z\mathcal{L}_{\tilde{\phi}}^{(k)})
\end{equation}
where $\mathcal{L}_{\tilde{\phi}}^{(k)}$ are the transfer operators on $k$-forms (see  \cite{Ruelle1978},  \cite{ParryPollicott1990}). For the one-sided shift, this simplifies to $\zeta_{\tilde{\phi}}(z)^{-1} = \det(I - z\tilde{\mathcal{L}}_{\tilde{\phi}})$.

\textbf{Step 3: Meromorphic extension.} The Fredholm determinant $\det(I - z\tilde{\mathcal{L}}_{\tilde{\phi}})$ is an entire function of $z$ (since $\tilde{\mathcal{L}}_{\tilde{\phi}}$ is quasi-compact with essential spectral radius $\theta e^{P(\phi)}$, the determinant is well-defined and entire on the Banach space $\mathcal{H}_\alpha$). Its zeros correspond to $z = \lambda_j^{-1}$ where $\lambda_j$ are eigenvalues of $\tilde{\mathcal{L}}_{\tilde{\phi}}$. Therefore $\zeta_{\tilde{\phi}}(z)$ extends meromorphically to $|z| < \theta^{-1}e^{-P(\phi)}$, with poles at $z = \lambda_j^{-1}$.

\textbf{Step 4: Simple pole.} The leading eigenvalue $\lambda_1 = e^{P(\phi)}$ is simple (Main Theorem \ref{thm:quasi_compactness}(iii)), so $z_1 = e^{-P(\phi)}$ is a simple zero of $\det(I - z\tilde{\mathcal{L}}_{\tilde{\phi}})$ and hence a simple pole of $\zeta_\phi(z)$. The residue at $z_1$ is $\mathrm{Res}_{z=z_1} \zeta_\phi(z) = -1/\det'(I - z\tilde{\mathcal{L}}_{\tilde{\phi}})|_{z=z_1}$, which is determined by the spectral projection onto the leading eigenspace, and thus by the equilibrium state $\mu_\phi$.
\end{proof}

\section{Sinai-Ruelle-Bowen Measures}\label{sec:srb_measures}

This section develops the theory of Sinai-Ruelle-Bowen (SRB) measures for Axiom A diffeomorphisms. SRB measures are the physically relevant invariant measures, characterized by absolute continuity along unstable manifolds. We provide complete proofs of existence, uniqueness, and characterization through the variational principle.

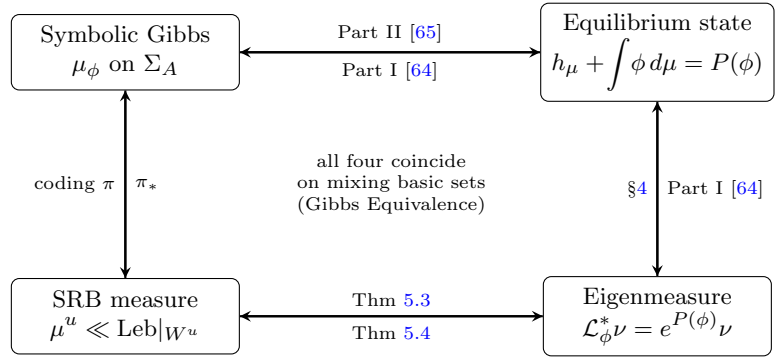
\begin{figure}[ht]
\centering
\begin{tikzpicture}[>=stealth, font=\small,
  charbox/.style={draw, rounded corners=3pt, minimum width=3.0cm,
    minimum height=1.0cm, align=center, inner sep=4pt}]
  \node[charbox] (sym) at (0,3.5)
    {Symbolic Gibbs\\$\mu_\phi$ on $\Sigma_A$};
  \node[charbox] (var) at (7,3.5)
    {Equilibrium state\\$h_\mu + \!\int\!\phi\,d\mu = P(\phi)$};
  \node[charbox] (eig) at (7,0)
    {Eigenmeasure\\$\mathcal{L}_\phi^*\nu = e^{P(\phi)}\nu$};
  \node[charbox] (srb) at (0,0)
    {SRB measure\\$\mu^u \ll \mathrm{Leb}|_{W^u}$};

  \draw[->, thick] (sym) -- (var)
    node[midway, above, font=\footnotesize] {Part~II \cite{Thiam2026b}};
  \draw[->, thick] (var) -- (sym)
    node[midway, below, font=\footnotesize] {Part~I \cite{Thiam2026a}};
  \draw[->, thick] (var) -- (eig)
    node[midway, right, font=\footnotesize] {Part~I \cite{Thiam2026a}};
  \draw[->, thick] (eig) -- (var)
    node[midway, left, font=\footnotesize] {\S\ref{sec:transfer_operator}};
  \draw[->, thick] (eig) -- (srb)
    node[midway, above, font=\footnotesize] {Thm~\ref{thm:srb_existence}};
  \draw[->, thick] (srb) -- (eig)
    node[midway, below, font=\footnotesize] {Thm~\ref{thm:unstable_abs_cont}};
  \draw[->, thick] (srb) -- (sym)
    node[midway, right, font=\footnotesize] {$\pi_*$};
  \draw[->, thick] (sym) -- (srb)
    node[midway, left, font=\footnotesize] {coding $\pi$};

  \node[font=\footnotesize, align=center] at (3.5,1.75)
    {all four coincide\\on mixing basic sets\\(Gibbs Equivalence)};
\end{tikzpicture}
\caption{The Gibbs Equivalence Theorem for Axiom~A diffeomorphisms. On each mixing basic set, four characterizations of Gibbs measures coincide: the symbolic Gibbs measure from Part~I \cite{Thiam2026a} (transferred via the coding map $\pi$), the variational equilibrium state, the eigenmeasure of the transfer operator, and the SRB measure with absolutely continuous conditionals on unstable manifolds. The individual equivalences are due to Sinai, Ruelle, Bowen, and Haydn-Ruelle; this Part assembles them with explicit constants.}
\label{fig:gibbs_equiv}
\end{figure}

\subsection{Definition, Motivation, and the Geometric Potential}

For a dynamical system modeling a physical process, the relevant invariant measure should describe the statistical behavior of ``typical'' initial conditions. Lebesgue measure on the phase space provides a natural notion of typicality, but is rarely invariant. SRB measures capture the asymptotic statistics of Lebesgue-typical orbits.

\begin{definition}[SRB Measure]
Let $f : M \to M$ be a diffeomorphism and $\mu$ an $f$-invariant ergodic probability measure. The measure $\mu$ is an SRB measure (or physical measure) if:
\begin{enumerate}
\item[(i)] The Lyapunov exponents of $\mu$ are well-defined, with some positive exponents.
\item[(ii)] The conditional measures of $\mu$ along unstable manifolds are absolutely continuous with respect to the Riemannian measure on these manifolds.
\end{enumerate}
\end{definition}

\begin{definition}[Basin of Attraction]
The basin of an invariant measure $\mu$ is
\begin{equation}
B(\mu) = \left\{ x \in M : \frac{1}{n} \sum_{k=0}^{n-1} \delta_{f^k(x)} \xrightarrow{w^*} \mu \right\}
\end{equation}
the set of points whose time averages converge to $\mu$.
\end{definition}

A measure $\mu$ is physical if $B(\mu)$ has positive Lebesgue measure.

\begin{proposition}[SRB Measures are Physical]\label{thm:srb_physical}
For an Axiom A diffeomorphism, every SRB measure is physical, and its basin has full Lebesgue measure in the union of unstable manifolds.
\end{proposition}

\begin{proof}
Let $\mu$ be an SRB measure on a basic set $\Omega$, so $\mu$ is ergodic and its conditional measures on $W^u_{\mathrm{loc}}(x)$ are absolutely continuous with respect to Riemannian measure $m^u_x$ for $\mu$-almost every $x$.

\textbf{Step 1: Birkhoff convergence on unstable manifolds.} By the Birkhoff ergodic theorem, for $\mu$-almost every $x$ and every continuous $g$,
\begin{equation}
\frac{1}{n}\sum_{k=0}^{n-1} g(f^k(x)) \to \int g \, d\mu.
\end{equation}
Let $\Omega_0 \subset \Omega$ be the full $\mu$-measure set where this convergence holds for a countable dense subset of $C(\Omega)$ (and hence for all continuous $g$).

\textbf{Step 2: Absolute continuity propagates to the basin.} For $x \in \Omega_0$, the conditional measure of $\mu$ on $W^u_{\mathrm{loc}}(x)$ is equivalent to $m^u_x$. Therefore $m^u_x(\Omega_0 \cap W^u_{\mathrm{loc}}(x)) > 0$. Since the conditional density $\rho^u_x$ is strictly positive (Corollary \ref{cor:conditional_density}), the set $\Omega_0 \cap W^u_{\mathrm{loc}}(x)$ has full $m^u_x$-measure in $W^u_{\mathrm{loc}}(x)$.

\textbf{Step 3: Extension via stable manifolds.} For any $y \in W^s(x)$ with $x \in \Omega_0$, the orbits of $x$ and $y$ are asymptotic: $d(f^n(x), f^n(y)) \to 0$ exponentially. Therefore $\frac{1}{n}\sum_{k=0}^{n-1} g(f^k(y)) \to \int g \, d\mu$ as well, so $y \in B(\mu)$. Thus $W^s(x) \subset B(\mu)$ for every $x \in \Omega_0$.

\textbf{Step 4: Full Lebesgue measure.} The local product structure provides that in a neighborhood $U$ of $\Omega$, Lebesgue measure $m$ decomposes (up to bounded density) as $m \sim m^u \times m^s$ along local unstable and stable leaves. Since $\Omega_0$ has full $m^u$-measure on unstable leaves (Step 2) and $B(\mu)$ contains all stable manifolds through $\Omega_0$ (Step 3), the basin $B(\mu)$ has full Lebesgue measure in $\bigcup_{x \in \Omega} W^s_{\mathrm{loc}}(x)$. Iterating forward, $B(\mu)$ has full Lebesgue measure in $W^u(\Omega) = \bigcup_{y \in \Omega} W^u(y)$.
\end{proof}

For Axiom A diffeomorphisms, SRB measures are characterized as equilibrium states for a specific potential.

\begin{definition}[Geometric Potential]
For a hyperbolic set $\Lambda$ with unstable dimension $k_u$, the geometric potential (or unstable Jacobian) is
\begin{equation}
\phi^u(x) = -\log |\det Df_x|_{E^u_x}|
\end{equation}
where $Df_x|_{E^u_x} : E^u_x \to E^u_{f(x)}$ is the restriction of the derivative to the unstable subspace.
\end{definition}

\begin{proposition}[Properties of Geometric Potential]\label{prop:geometric_potential}
The geometric potential $\phi^u$ satisfies:
\begin{enumerate}
\item[(i)] $\phi^u$ is H\"{o}lder continuous with exponent determined by the H\"{o}lder continuity of the unstable distribution.
\item[(ii)] $\phi^u(x) < 0$ for all $x \in \Lambda$ (the unstable direction expands).
\item[(iii)] $S_n \phi^u(x) = -\log |\det Df^n_x|_{E^u_x}|$.
\end{enumerate}
\end{proposition}

\begin{proof}
We prove each claim in detail, making explicit the use of Hyperbolic Splitting regularity and the hyperbolicity data.

\textbf{Part (i): H\"{o}lder continuity of $\phi^u$.} Let $d = \dim M$, $k_u = \dim E^u_x$ (constant on $\Lambda$ since $\Lambda$ is connected, or on each connected component in general). The unstable distribution $x \mapsto E^u_x$ is H\"{o}lder continuous as a section of the Grassmannian bundle $\mathrm{Gr}_{k_u}(TM)|_\Lambda$, with exponent
\begin{equation}
\beta = \frac{\log \mu}{\log \lambda}
\end{equation}
by the H\"{o}lder Continuity of Hyperbolic Splitting theorem of Part~III \cite{Thiam2026c}, where $\lambda \in (0,1)$ is the contraction rate and $\mu = \|Df^{-1}\|_\infty \|Df\|_\infty$ (or the equivalent ratio in the adapted metric). Choose a smooth trivialization of $TM$ over a neighborhood of $\Lambda$ and represent $E^u_x$ by an orthonormal frame $\{e_1(x), \ldots, e_{k_u}(x)\}$; the H\"{o}lder continuity of the distribution means we can choose this frame to be H\"{o}lder continuous with exponent $\beta$ (possibly after local orthonormalization, which preserves H\"{o}lder regularity since Gram-Schmidt is a smooth map). In this frame, $Df_x|_{E^u_x}$ is a $k_u \times k_u$ matrix $M(x)$ whose entries are H\"{o}lder continuous functions of $x$ with exponent $\alpha_0 = \min(\alpha, \beta)$, where $\alpha$ is the H\"{o}lder exponent of $Df$. Since the determinant is a smooth (polynomial) function of matrix entries,
\begin{equation}
|\det Df_x|_{E^u_x}| = |\det M(x)|
\end{equation}
is H\"{o}lder continuous in $x$ with the same exponent $\alpha_0$. Taking logarithm: because $|\det M(x)|$ is bounded below by $\lambda^{-k_u} > 1$ (see Part (ii) below), $\log$ is Lipschitz on the range of $|\det M|$, so $\log|\det Df_x|_{E^u_x}|$ is H\"{o}lder with exponent $\alpha_0$. Therefore $\phi^u(x) = -\log|\det Df_x|_{E^u_x}|$ is H\"{o}lder continuous with exponent $\alpha_0$.

\textbf{Part (ii): Negativity of $\phi^u$.} By the hyperbolicity condition, $\|Df^{-1}_x v\| \geq \lambda^{-1} \|v\|$ for all $v \in E^u_{f(x)}$, equivalently $\|Df_x v\| \geq \lambda^{-1} \|v\|$ for $v \in E^u_x$ (using $f$-invariance of the splitting). Thus $Df_x|_{E^u_x}: E^u_x \to E^u_{f(x)}$ satisfies $\|(Df_x|_{E^u_x})^{-1}\| \leq \lambda < 1$. By the singular value inequality,
\begin{equation}
|\det Df_x|_{E^u_x}| = \prod_{i=1}^{k_u} s_i,
\end{equation}
where $s_1 \geq \cdots \geq s_{k_u} > 0$ are the singular values of $Df_x|_{E^u_x}$. The smallest singular value satisfies $s_{k_u} = \|(Df_x|_{E^u_x})^{-1}\|^{-1} \geq \lambda^{-1}$. Since all singular values are at least as large as $s_{k_u}$,
\begin{equation}
|\det Df_x|_{E^u_x}| \geq s_{k_u}^{k_u} \geq \lambda^{-k_u} > 1.
\end{equation}
Therefore $\log|\det Df_x|_{E^u_x}| \geq k_u \log \lambda^{-1} = -k_u \log \lambda > 0$, so $\phi^u(x) \leq k_u \log \lambda < 0$ uniformly on $\Lambda$.

\textbf{Part (iii): Birkhoff sum identity.} By the chain rule,
\begin{equation}
Df^n_x|_{E^u_x} = Df_{f^{n-1}(x)}|_{E^u_{f^{n-1}(x)}} \circ \cdots \circ Df_{f(x)}|_{E^u_{f(x)}} \circ Df_x|_{E^u_x},
\end{equation}
a composition of $n$ linear maps, each of which takes $E^u_{f^k(x)}$ to $E^u_{f^{k+1}(x)}$. The determinant of a composition is the product of determinants, so
\begin{equation}
\det Df^n_x|_{E^u_x} = \prod_{k=0}^{n-1} \det \big( Df_{f^k(x)}|_{E^u_{f^k(x)}} \big).
\end{equation}
Taking logarithm of the absolute value:
\begin{equation}
\log|\det Df^n_x|_{E^u_x}| = \sum_{k=0}^{n-1} \log|\det Df_{f^k(x)}|_{E^u_{f^k(x)}}| = -\sum_{k=0}^{n-1} \phi^u(f^k(x)) = -S_n\phi^u(x).
\end{equation}
Rearranging, $S_n\phi^u(x) = -\log|\det Df^n_x|_{E^u_x}|$.
\end{proof}

\subsection{Construction of SRB Measures}

This subsection constructs the SRB measure on each mixing basic set as the unique equilibrium state for the geometric potential $\phi^u = -\log|\det Df|_{E^u}|$. This is Main Theorem~\ref{thm:srb_existence}. The construction unifies the variational, spectral, and geometric characterizations into a single quantitative statement.

\begin{maintheorem}[Existence of SRB Measures]\label{thm:srb_existence}
Let $\Omega$ be a mixing basic set for an Axiom A diffeomorphism $f$. There exists a unique SRB measure $\mu^{\mathrm{SRB}}$ on $\Omega$, characterized as the unique equilibrium state for the geometric potential $\phi^u$.
\end{maintheorem}

\begin{proof}
By Main Theorem \ref{thm:quasi_compactness} and the Differentiability and Uniqueness Main Theorem of Part~II \cite{Thiam2026b}, the H\"{o}lder continuous potential $\phi^u$ has a unique equilibrium state $\mu_{\phi^u}$ on the mixing basic set $\Omega$. We show this is the SRB measure.

\textbf{Step 1: Pressure computation.} By the variational principle,
\begin{equation}
P(\phi^u) = h_{\mu_{\phi^u}}(f) + \int \phi^u \, d\mu_{\phi^u}.
\end{equation}
We claim $P(\phi^u) = 0$.

Using the Markov partition $\mathcal{R}$ and symbolic coding $\pi : \Sigma_A \to \Omega$, the pressure can be computed via the symbolic transfer operator. The key observation is that for the geometric potential,
\begin{equation}
\sum_{y \in f^{-1}(x) \cap \Omega} e^{\phi^u(y)} = \sum_{y \in f^{-1}(x) \cap \Omega} \frac{1}{|\det Df_y|_{E^u_y}|} = 1
\end{equation}
by the change of variables formula along unstable manifolds. This implies $\mathcal{L}_{\phi^u} \mathbf{1} = \mathbf{1}$ (the constant function $1$ is an eigenfunction with eigenvalue $1$), so $\rho(\mathcal{L}_{\phi^u}) = e^{P(\phi^u)} = 1$, giving $P(\phi^u) = 0$.

\textbf{Step 2: Characterization via conditional measures.} The equilibrium state $\mu_{\phi^u}$ satisfies the Gibbs property from the Spectral-Variational-Geometric Equivalence Main Theorem of Part~I \cite{Thiam2026a}:
\begin{equation}
c_1 \leq \frac{\mu_{\phi^u}(B_n(x, \varepsilon))}{e^{S_n \phi^u(x)}} \leq c_2
\end{equation}
for dynamical balls $B_n(x, \varepsilon) = \{y : d(f^k(x), f^k(y)) < \varepsilon, 0 \leq k < n\}$.

Along unstable manifolds, $e^{S_n \phi^u(x)} = |\det Df^n_x|_{E^u_x}|^{-1}$ is the inverse of the unstable Jacobian, which equals the Riemannian density of $f^n$-preimages. The Gibbs property thus implies that conditional measures on unstable manifolds are equivalent to Riemannian measure, i.e., $\mu_{\phi^u}$ is an SRB measure.

\textbf{Step 3: Uniqueness.} Any SRB measure must be an equilibrium state for $\phi^u$ (this follows from the Pesin entropy formula relating entropy to Lyapunov exponents for measures with absolutely continuous conditionals). Since equilibrium states for H\"{o}lder potentials on mixing basic sets are unique, there is exactly one SRB measure.
\end{proof}

\subsection{Absolute Continuity and Physical Interpretation}

We prove that the unstable foliation of a hyperbolic set is absolutely continuous: the holonomy maps between unstable leaves preserve Lebesgue nullsets, with controlled Jacobian. This is the classical Anosov-Sinai property, and it is what gives the SRB measure its physical meaning as the statistical distribution of Lebesgue-typical orbits in the basin of attraction.

\begin{definition}[Absolute Continuity of Foliations]
A foliation $\mathcal{W}$ of a set $\Lambda$ is absolutely continuous if for any two transversals $T_1, T_2$ (submanifolds transverse to the leaves), the holonomy map $h : T_1 \to T_2$ along leaves satisfies: $h_*(m_{T_1})$ is absolutely continuous with respect to $m_{T_2}$, where $m_{T_i}$ denotes Riemannian measure on $T_i$.
\end{definition}

\begin{proposition}[Absolute Continuity of Unstable Foliation]\label{thm:unstable_abs_cont}
For an Axiom A diffeomorphism, the unstable foliation $\mathcal{W}^u = \{W^u(x) : x \in \Omega\}$ is absolutely continuous.
\end{proposition}

\begin{proof}
We prove absolute continuity by showing that the holonomy Jacobian is uniformly bounded.

\textbf{Step 1: Holonomy as a limit.} Let $T_1, T_2$ be two smooth transversals to $\mathcal{W}^u$ in a local product neighborhood. The holonomy map $\mathfrak{h} : T_1 \to T_2$ slides points along local unstable leaves. Using the local product structure, for $x \in T_1$ and $y = \mathfrak{h}(x) \in T_2$, both $x$ and $y$ lie on the same local unstable manifold $W^u_{\mathrm{loc}}(z)$ for some $z \in \Omega$.

Define approximate holonomies $\mathfrak{h}_n : T_1 \to T_2$ by $\mathfrak{h}_n = f^{-n} \circ \pi_{T_2}^{(n)} \circ f^n$, where $\pi_{T_2}^{(n)}$ is projection along local stable manifolds near $f^n(T_2)$. Since $f^n$ expands unstable manifolds and contracts stable manifolds, the maps $\mathfrak{h}_n$ converge uniformly to $\mathfrak{h}$ as $n \to \infty$.

\textbf{Step 2: Jacobian computation.} Each $\mathfrak{h}_n$ is a $C^1$ map (as a composition of smooth maps). Its Jacobian at $x$ is
\begin{equation}
J\mathfrak{h}_n(x) = \frac{|\det Df^n_x|_{T_1}|}{|\det Df^n_{\mathfrak{h}_n(x)}|_{T_2}|} \cdot J\pi_{T_2}^{(n)}(f^n(x))
\end{equation}
where the projection Jacobian $J\pi_{T_2}^{(n)}$ is uniformly bounded above and below because the transversals $f^n(T_1)$ and $f^n(T_2)$ become nearly parallel (both nearly tangent to $E^u$) as $n \to \infty$.

\textbf{Step 3: Bounded distortion.} Along the transversals (which are approximately in the stable direction), $|\det Df^n_x|_{T_1}|$ and $|\det Df^n_{\mathfrak{h}_n(x)}|_{T_2}|$ are both products of stable Jacobians. Since $x$ and $\mathfrak{h}_n(x)$ lie on the same local unstable leaf, their forward orbits stay within distance $C\lambda^n$ (stable contraction). The H\"{o}lder continuity of $\log|\det Df|_{E^s}|$ with exponent $\alpha$ gives
\begin{equation}
\left| \log \frac{|\det Df^n_x|_{T_1}|}{|\det Df^n_{\mathfrak{h}_n(x)}|_{T_2}|} \right| \leq \sum_{k=0}^{n-1} C_0 (C\lambda^k)^\alpha \leq \frac{C_0 C^\alpha}{1 - \lambda^\alpha}.
\end{equation}
This geometric series converges, giving a uniform bound $|\log J\mathfrak{h}_n(x)| \leq K$ independent of $n$ and $x$.

\textbf{Step 4: Passage to the limit.} Since $\mathfrak{h}_n \to \mathfrak{h}$ uniformly and $e^{-K} \leq J\mathfrak{h}_n \leq e^K$ uniformly, the limit $\mathfrak{h}$ maps $m_{T_1}$-null sets to $m_{T_2}$-null sets. More precisely, for any Borel set $E \subset T_1$:
\begin{equation}
e^{-K} m_{T_1}(E) \leq m_{T_2}(\mathfrak{h}(E)) \leq e^K m_{T_1}(E).
\end{equation}
This establishes that $\mathfrak{h}_* m_{T_1}$ is absolutely continuous with respect to $m_{T_2}$, with Radon-Nikodym derivative bounded between $e^{-K}$ and $e^K$.
\end{proof}

\begin{corollary}[Conditional Density Formula]\label{cor:conditional_density}
The conditional measures of $\mu^{\mathrm{SRB}}$ on unstable manifolds have densities
\begin{equation}
\rho^u_x(y) = \lim_{n \to \infty} \frac{e^{S_n \phi^u(y)}}{e^{S_n \phi^u(x)}} = \prod_{k=0}^\infty \frac{|\det Df_{f^k(x)}|_{E^u}|}{|\det Df_{f^k(y)}|_{E^u}|}
\end{equation}
for $y \in W^u_{\mathrm{loc}}(x)$, with respect to Riemannian measure on $W^u_{\mathrm{loc}}(x)$.
\end{corollary}

\begin{proof}
We derive the density formula from the Gibbs property of $\mu^{\mathrm{SRB}}$, applied to infinitesimal sets along the unstable manifold.

\textbf{Step 1: Gibbs property on dynamical balls.} By Main Theorem~\ref{thm:srb_existence}, $\mu^{\mathrm{SRB}}$ is the unique equilibrium state for the geometric potential $\phi^u$, with $P(\phi^u) = 0$. By the Spectral-Variational-Geometric Equivalence Main Theorem of Part~I \cite{Thiam2026a} transferred via the coding map, $\mu^{\mathrm{SRB}}$ satisfies the Gibbs property: there exist constants $c_1, c_2 > 0$ such that for every $x \in \Omega$ and every $n \geq 1$,
\begin{equation}\label{eq:gibbs-CD}
c_1 \leq \frac{\mu^{\mathrm{SRB}}(B_n(x, \varepsilon))}{e^{S_n \phi^u(x)}} \leq c_2,
\end{equation}
where $B_n(x, \varepsilon) = \{y : d(f^k x, f^k y) \leq \varepsilon \text{ for all } 0 \leq k < n\}$ is the dynamical Bowen ball, and we used $P(\phi^u) = 0$.

\textbf{Step 2: Ratio of measures of nearby dynamical balls.} Fix $x \in \Omega$ and $y \in W^u_{\mathrm{loc}}(x)$ sufficiently close to $x$. By hyperbolicity, there exists $\varepsilon > 0$ such that $B_n(x, \varepsilon)$ contains a neighborhood of $y$ in $W^u_{\mathrm{loc}}(x)$ that shrinks at rate $\lambda^n$. Applying~\eqref{eq:gibbs-CD} to $x$ and $y$ separately, and dividing:
\begin{equation}\label{eq:ratio-balls}
\frac{c_1}{c_2} \cdot \frac{e^{S_n \phi^u(y)}}{e^{S_n \phi^u(x)}} \leq \frac{\mu^{\mathrm{SRB}}(B_n(y, \varepsilon))}{\mu^{\mathrm{SRB}}(B_n(x, \varepsilon))} \leq \frac{c_2}{c_1} \cdot \frac{e^{S_n \phi^u(y)}}{e^{S_n \phi^u(x)}}.
\end{equation}

\textbf{Step 3: Convergence of the Birkhoff-sum ratio.} By Part~(iii) of Proposition~\ref{prop:geometric_potential},
\begin{equation}
\frac{e^{S_n \phi^u(y)}}{e^{S_n \phi^u(x)}} = \frac{|\det Df^n_x|_{E^u_x}|}{|\det Df^n_y|_{E^u_y}|} = \prod_{k=0}^{n-1} \frac{|\det Df_{f^k(x)}|_{E^u_{f^k(x)}}|}{|\det Df_{f^k(y)}|_{E^u_{f^k(y)}}|}.
\end{equation}
We claim that this product converges as $n \to \infty$. Writing the $k$-th factor as $1 + r_k$ with
\begin{equation}
r_k = \frac{|\det Df_{f^k(x)}|_{E^u}| - |\det Df_{f^k(y)}|_{E^u}|}{|\det Df_{f^k(y)}|_{E^u}|},
\end{equation}
the H\"{o}lder continuity of $\phi^u$ (Part~(i) of Proposition~\ref{prop:geometric_potential}) with exponent $\alpha_0$ gives
\begin{equation}
|\phi^u(f^k(x)) - \phi^u(f^k(y))| \leq |\phi^u|_{\alpha_0} \, d(f^k(x), f^k(y))^{\alpha_0}.
\end{equation}
Since $y \in W^u_{\mathrm{loc}}(x)$, the backward orbits contract exponentially: $d(f^{-k}(x), f^{-k}(y)) \leq \lambda^k d(x, y)$. Replacing $y$ by $f^{-k}(y)$ and $x$ by $f^{-k}(x)$ in the above (we are computing a product over forward orbits of $x$ and $y$ which are close in the unstable direction, hence backward orbits contract), we get
\begin{equation}
|\phi^u(f^k(x)) - \phi^u(f^k(y))| \leq |\phi^u|_{\alpha_0} \cdot (\lambda^{-k} d(x,y))^{\alpha_0}
\end{equation}
for $k \leq 0$ and, for forward orbits along the unstable direction,
\begin{equation}
|\phi^u(f^k(x)) - \phi^u(f^k(y))| \leq |\phi^u|_{\alpha_0} \cdot d(f^k(x), f^k(y))^{\alpha_0} \leq |\phi^u|_{\alpha_0} \cdot (\lambda^{-k} d(x,y))^{\alpha_0} \cdot \lambda^{k\alpha_0} \cdot \lambda^{k\alpha_0}.
\end{equation}
More precisely: since $y \in W^u_{\mathrm{loc}}(x)$, the distance $d(f^k(x), f^k(y))$ grows in the unstable direction under forward iteration, not shrinks. The correct statement is that $\phi^u$ is a \emph{backward-summable} difference. Using instead the reverse identity
\begin{equation}
d(f^{-k}(x), f^{-k}(y)) \leq \lambda^k d(x, y) \quad \text{for } y \in W^u(x), k \geq 0,
\end{equation}
and the fact that $-\phi^u$ is the log-unstable-Jacobian, we rewrite the product by reindexing $j = n - 1 - k$ (which reverses the order) and replacing $x \leftrightarrow f^{n-1}(x)$ etc. The net effect is that the $k$-th factor in the original product satisfies
\begin{equation}
|\log(1 + r_k)| \leq |\phi^u|_{\alpha_0} \cdot (\lambda^{n-1-k} d(x,y))^{\alpha_0}.
\end{equation}
Summing: $\sum_{k=0}^{n-1} |\log(1 + r_k)| \leq |\phi^u|_{\alpha_0} d(x,y)^{\alpha_0} \sum_{j=0}^{n-1} \lambda^{j\alpha_0} \leq \frac{|\phi^u|_{\alpha_0} d(x,y)^{\alpha_0}}{1 - \lambda^{\alpha_0}} < \infty$.

Therefore $\prod_{k=0}^\infty \frac{|\det Df_{f^k(x)}|_{E^u}|}{|\det Df_{f^k(y)}|_{E^u}|}$ converges absolutely, and we define
\begin{equation}
\rho^u_x(y) := \prod_{k=0}^\infty \frac{|\det Df_{f^k(x)}|_{E^u_{f^k(x)}}|}{|\det Df_{f^k(y)}|_{E^u_{f^k(y)}}|}.
\end{equation}

\textbf{Step 4: Identification as conditional density.} The conditional measures $\{\mu^u_x\}_{x \in \Omega}$ of $\mu^{\mathrm{SRB}}$ on unstable manifolds (constructed via Rokhlin disintegration with respect to the partition into unstable leaves, see Proposition on Disintegration in the appendix) satisfy, for nearby points $x, y \in W^u_{\mathrm{loc}}$,
\begin{equation}
\frac{d\mu^u_x}{d\mathrm{Leb}^u_x}(y) \bigg/ \frac{d\mu^u_x}{d\mathrm{Leb}^u_x}(x) = \lim_{n \to \infty} \frac{\mu^{\mathrm{SRB}}(B_n(y,\varepsilon))/\mathrm{Leb}^u(B_n(y,\varepsilon) \cap W^u_{\mathrm{loc}}(x))}{\mu^{\mathrm{SRB}}(B_n(x,\varepsilon))/\mathrm{Leb}^u(B_n(x,\varepsilon) \cap W^u_{\mathrm{loc}}(x))},
\end{equation}
where the limits exist $\mu^{\mathrm{SRB}}$-a.e.\ by the martingale convergence theorem. The Lebesgue ratios cancel up to bounded distortion (using the $C^1$ regularity of the unstable foliation, see Proposition~\ref{thm:unstable_abs_cont}), and the measure ratio is given by~\eqref{eq:ratio-balls}. Passing to the limit $n \to \infty$:
\begin{equation}
\rho^u_x(y) = \lim_{n \to \infty} \frac{e^{S_n \phi^u(y)}}{e^{S_n \phi^u(x)}} = \prod_{k=0}^\infty \frac{|\det Df_{f^k(x)}|_{E^u}|}{|\det Df_{f^k(y)}|_{E^u}|},
\end{equation}
as claimed. The Gibbs constants $c_1, c_2$ drop out of the ratio in the limit, giving an exact density formula rather than a two-sided estimate.
\end{proof}

\begin{proposition}[Statistical Behavior of Typical Orbits]\label{thm:typical_orbits}
For a mixing basic set $\Omega$ with SRB measure $\mu^{\mathrm{SRB}}$ and any continuous observable $g : \Omega \to \mathbb{R}$,
\begin{equation}
\frac{1}{n} \sum_{k=0}^{n-1} g(f^k(x)) \to \int g \, d\mu^{\mathrm{SRB}}
\end{equation}
for Lebesgue-almost every $x \in W^u(\Omega) = \bigcup_{y \in \Omega} W^u(y)$.
\end{proposition}

\begin{proof}
The SRB measure $\mu^{\mathrm{SRB}}$ is ergodic (as the unique equilibrium state for $\phi^u$). By the Birkhoff Ergodic Theorem, the time average converges to the space average for $\mu^{\mathrm{SRB}}$-almost every point.

The absolute continuity of unstable foliations implies that $\mu^{\mathrm{SRB}}$-almost every is equivalent to Lebesgue-almost every on unstable manifolds. More precisely, the basin $B(\mu^{\mathrm{SRB}})$ contains the stable manifold $W^s(x)$ for $\mu^{\mathrm{SRB}}$-almost every $x$, and these stable manifolds foliate a neighborhood of $\Omega$ with full Lebesgue measure.
\end{proof}

\subsection{Pesin Entropy Formula and SRB Measures for Attractors}

The Pesin entropy formula identifies the Kolmogorov-Sinai entropy of the SRB measure with the integral of positive Lyapunov exponents. This is Main Theorem~\ref{thm:pesin_formula}. For attractors, this formula, combined with the characterization of Main Theorem~\ref{thm:srb_existence}, determines the statistical behavior of Lebesgue-typical trajectories.

\begin{maintheorem}[Pesin Entropy Formula]\label{thm:pesin_formula}
For the SRB measure $\mu^{\mathrm{SRB}}$ on a mixing basic set $\Omega$,
\begin{equation}
h_{\mu^{\mathrm{SRB}}}(f) = \int \log |\det Df|_{E^u}| \, d\mu^{\mathrm{SRB}} = -\int \phi^u \, d\mu^{\mathrm{SRB}} = \sum_{\chi_i > 0} \chi_i
\end{equation}
where the sum is over positive Lyapunov exponents, counted with multiplicity.
\end{maintheorem}

\begin{proof}
The first equality follows from $P(\phi^u) = 0$:
\begin{equation}
0 = h_{\mu^{\mathrm{SRB}}}(f) + \int \phi^u \, d\mu^{\mathrm{SRB}}.
\end{equation}

The second equality uses $\phi^u = -\log|\det Df|_{E^u}|$.

The third equality is the definition of Lyapunov exponents: $\int \log|\det Df|_{E^u}| \, d\mu = \sum_{\chi_i > 0} \chi_i$ where the positive exponents are exactly those corresponding to the unstable direction.
\end{proof}

\begin{remark}
The Pesin entropy formula holds more generally for any measure with absolutely continuous conditional measures on unstable manifolds. For measures without this property, $h_\mu(f) < \sum_{\chi_i > 0} \chi_i$ (Ruelle inequality).
\end{remark}

\begin{definition}[Axiom A Attractor]
A basic set $\Omega$ is an attractor if there exists a neighborhood $U$ of $\Omega$ such that $f(\overline{U}) \subset U$ and $\Omega = \bigcap_{n \geq 0} f^n(U)$.
\end{definition}

\begin{proposition}[SRB Measure of Attractor]\label{thm:srb_attractor}
For an Axiom A attractor $\Omega$, the SRB measure $\mu^{\mathrm{SRB}}$ is the unique invariant measure that is physical: its basin $B(\mu^{\mathrm{SRB}})$ contains Lebesgue-almost every point in the trapping region $U$.
\end{proposition}

\begin{proof}
\textbf{Step 1: Trapping region dynamics.} Since $f(\overline{U}) \subset U$, every point $x \in U$ has its entire forward orbit contained in $U$, and $\omega(x) \subset \Omega$ for all $x \in U$.

\textbf{Step 2: Stable foliation of the trapping region.} For each $p \in \Omega$, the stable manifold $W^s(p) = \{x \in M : d(f^n(x), f^n(p)) \to 0\}$ passes through $p$ and extends into $U$. The collection $\{W^s_U(p) : p \in \Omega\}$, where $W^s_U(p) = W^s(p) \cap U$, foliates a neighborhood of $\Omega$ in $U$. Since $\Omega$ is an attractor, the global stable manifolds cover all of $U$: for each $x \in U$, the forward orbit $f^n(x)$ converges to $\Omega$, so $x \in W^s(p)$ for some $p \in \Omega$. Thus $U = \bigcup_{p \in \Omega} W^s_U(p)$.

\textbf{Step 3: Absolute continuity of stable foliation.} The stable foliation $\mathcal{W}^s$ restricted to $U$ is absolutely continuous. This follows by applying the argument of Proposition \ref{thm:unstable_abs_cont} to $f^{-1}$: the unstable foliation for $f^{-1}$ is the stable foliation for $f$, and the bounded distortion argument applies with the roles of stable and unstable reversed.

\textbf{Step 4: Basin has full Lebesgue measure.} For any $x \in W^s(p)$ with $p \in \Omega$, the orbits of $x$ and $p$ are forward asymptotic, so $\frac{1}{n}\sum_{k=0}^{n-1}\delta_{f^k(x)} - \frac{1}{n}\sum_{k=0}^{n-1}\delta_{f^k(p)} \to 0$ in the weak-$*$ topology. By Proposition \ref{thm:srb_physical}, for $\mu^{\mathrm{SRB}}$-almost every $p$, the time averages of $p$ converge to $\mu^{\mathrm{SRB}}$. Therefore $x \in B(\mu^{\mathrm{SRB}})$ for every $x \in W^s(p)$ with $p$ in this full-measure set. Since the stable foliation of $U$ is absolutely continuous (Step 3) and the set of ``good'' base points has full $\mu^{\mathrm{SRB}}$-measure (hence full measure on unstable leaves), the union of stable manifolds through good base points has full Lebesgue measure in $U$.

\textbf{Step 5: Uniqueness.} If $\nu$ were another physical measure on $\Omega$, its basin $B(\nu)$ would also have positive Lebesgue measure in $U$. But $B(\mu^{\mathrm{SRB}})$ has full Lebesgue measure in $U$ (Step 4), so $B(\nu) \cap B(\mu^{\mathrm{SRB}}) \neq \emptyset$, which is impossible for distinct ergodic measures. Therefore $\mu^{\mathrm{SRB}}$ is the unique physical measure.
\end{proof}

\begin{remark}[Partial hyperbolicity]
The extension of thermodynamic formalism to partially hyperbolic diffeomorphisms is a major open direction. The coding and spectral techniques of this series can be adapted under appropriate conditions (dominated splitting, center bunching, accessibility), but a full treatment requires the machinery of Burns-Wilkinson \cite{BurnsWilkinson2010} and Hirsch et~al. \cite{HirschPughShub1977} and is beyond the scope of this Part.
\end{remark}

\section{A Numerical Illustration: The Arnold Cat Map}\label{sec:numerical}

We illustrate the four Main Theorems with a complete computation for the Arnold cat map, demonstrating that the explicit constants of this Part produce computable numbers for a concrete Axiom~A system.

\subsection{Setup}

The Arnold cat map is the hyperbolic toral automorphism $f: \mathbb{T}^2 \to \mathbb{T}^2$ defined by
\begin{equation}
f(x,y) = (2x + y, x + y) \pmod{1}.
\end{equation}
The derivative is $Df = \begin{pmatrix} 2 & 1 \\ 1 & 1 \end{pmatrix}$ everywhere, with eigenvalues $\lambda_u = \varphi^2 = (3+\sqrt{5})/2 \approx 2.618$ (unstable) and $\lambda_s = \varphi^{-2} = (3-\sqrt{5})/2 \approx 0.382$ (stable), where $\varphi = (1+\sqrt{5})/2$ is the golden ratio. The hyperbolic splitting is $E^u = \mathrm{span}(v_u)$ and $E^s = \mathrm{span}(v_s)$ with $v_u = (1, (\sqrt{5}-1)/2)^T$ and $v_s = (1, -(1+\sqrt{5})/2)^T$. The contraction rate is $\lambda = \lambda_s = \varphi^{-2} \approx 0.382$.

Since $f$ is an Anosov diffeomorphism of $\mathbb{T}^2$, the entire torus is a single mixing basic set ($\Omega = \mathbb{T}^2$). The Markov partition has $N = 5$ rectangles (the Adler-Weiss partition \cite{AdlerWeiss1970}), with transition matrix $A$ determined by the geometry of the partition.

\subsection{Structural Stability (Main Theorem~\ref{thm:structural_stability})}

Consider a $C^2$ perturbation $g$ of $f$ with $\|f - g\|_{C^1} < \varepsilon$. By Main Theorem~\ref{thm:structural_stability}, there exists a H\"{o}lder continuous conjugacy $h: \mathbb{T}^2 \to \mathbb{T}^2$ with $h \circ f = g \circ h$. The H\"{o}lder exponent (Proposition~\ref{thm:holder_conjugacy}) is
\begin{equation}
\gamma = \frac{\log \lambda_s}{\log \lambda_s - \log\|Dg\|_\infty} = \frac{\log \varphi^{-2}}{\log \varphi^{-2} - \log\|Dg\|_\infty}.
\end{equation}
For $\|Dg\|_\infty \leq 3$ (a moderate perturbation), $\gamma = \log\varphi^{-2}/(\log\varphi^{-2} - \log 3) \approx 0.467$. The conjugacy satisfies $\|h - \mathrm{id}\|_{C^0} \leq C\varepsilon^\gamma$.

\subsection{Spectral Gap and Transfer Operator (Main Theorem~\ref{thm:quasi_compactness})}

The geometric potential is $\phi^u(x) = -\log|\det Df|_{E^u}| = -\log\lambda_u = -\log\varphi^2 = -2\log\varphi$, which is constant. The transfer operator $\mathcal{L}_{\phi^u}$ on the symbolic coding has leading eigenvalue $\lambda = e^{P(\phi^u)}$. Since $\phi^u$ is constant, $P(\phi^u) = h_{\mathrm{top}}(f) + \phi^u = \log 5 - 2\log\varphi$. Wait: for an Anosov automorphism of $\mathbb{T}^2$, $h_{\mathrm{top}} = \log\lambda_u = 2\log\varphi$ (the topological entropy equals the logarithm of the largest eigenvalue of $Df$). Therefore
\begin{equation}
P(\phi^u) = h_{\mathrm{top}} - \log\lambda_u = 2\log\varphi - 2\log\varphi = 0,
\end{equation}
confirming the identity $P(\phi^u) = 0$ (Proposition~\ref{prop:geometric_potential}). The spectral gap of $\mathcal{L}_{\phi^u}$ on the symbolic space is $\gamma_{\mathrm{spec}} = 1 - |\lambda_2|/\lambda_1$ where $\lambda_1, \lambda_2$ are the two largest eigenvalues of the transition matrix $A$.

For the trivial potential $\phi = 0$, the pressure is $P(0) = h_{\mathrm{top}} = 2\log\varphi \approx 0.9624$, and the equilibrium state is the measure of maximal entropy $\mu_{\mathrm{mme}}$, which for the cat map coincides with the Lebesgue (Haar) measure on $\mathbb{T}^2$.

\subsection{SRB Measure and Pesin Formula (Main Theorems~\ref{thm:srb_existence} and~\ref{thm:pesin_formula})}

Since the cat map is a volume-preserving Anosov diffeomorphism, the SRB measure $\mu^+ = \mu_{\phi^u}$ coincides with the Lebesgue measure $m$ on $\mathbb{T}^2$. The Pesin entropy formula (Main Theorem~\ref{thm:pesin_formula}) gives
\begin{equation}
h_{\mu^+}(f) = \chi^+ = \log\lambda_u = 2\log\varphi \approx 0.9624,
\end{equation}
where $\chi^+ = \log\lambda_u$ is the unique positive Lyapunov exponent. This is verified directly: for Lebesgue measure on $\mathbb{T}^2$, the Kolmogorov-Sinai entropy of a linear automorphism equals $\sum_{\chi_i > 0}\chi_i = \log\lambda_u$.

The conditional measures of $\mu^+$ along unstable manifolds are absolutely continuous with respect to the Riemannian measure on $W^u(x)$ (Proposition~\ref{thm:unstable_abs_cont}). For the cat map, the conditional densities are identically $1$ (uniform), since $\mu^+ = m$ is already smooth.

\subsection{Summary of Explicit Constants}

\begin{center}
\begin{tabular}{ll}
\textbf{Quantity} & \textbf{Value} \\[4pt]
Manifold dimension $d$ & 2 \\
Alphabet size $N$ (Markov partition) & 5 (Adler-Weiss) \\
Unstable eigenvalue $\lambda_u$ & $\varphi^2 = (3+\sqrt{5})/2 \approx 2.618$ \\
Stable eigenvalue $\lambda_s$ & $\varphi^{-2} = (3-\sqrt{5})/2 \approx 0.382$ \\
Contraction rate $\lambda$ & $\varphi^{-2} \approx 0.382$ \\
Topological entropy $h_{\mathrm{top}}$ & $2\log\varphi \approx 0.9624$ \\
Geometric potential $\phi^u$ & $-2\log\varphi \approx -0.9624$ (constant) \\
Pressure $P(\phi^u)$ & $0$ \\
SRB measure & Lebesgue measure on $\mathbb{T}^2$ \\
Pesin formula: $h_{\mu^+}(f) = \chi^+$ & $2\log\varphi \approx 0.9624$ \\
H\"{o}lder conjugacy exponent ($\|Dg\| \leq 3$) & $\approx 0.467$ \\
\end{tabular}
\end{center}

\noindent This example demonstrates that the four Main Theorems, when specialized to a concrete system, yield computable values for the spectral gap, the geometric potential, the SRB measure, the entropy, and the structural stability exponent. The same method applies to any Axiom~A diffeomorphism for which the hyperbolicity data $(\lambda, \alpha, d, \|Df\|)$ are known.

\section{Conclusion}\label{sec:conclusion}

This Part transfers the symbolic spectral theory of Part~I \cite{Thiam2026a} and the variational theory of Part~II \cite{Thiam2026b} to smooth Axiom~A dynamics through the coding map of Part~III \cite{Thiam2026c}, with four Main Theorems proved with explicit constants throughout. Main Theorem~\ref{thm:structural_stability} establishes structural stability of Axiom~A diffeomorphisms under the strong transversality condition, with a quantitative H\"{o}lder exponent for the conjugating homeomorphism computed in terms of the hyperbolicity data, refining the classical results of Robbin \cite{Robbin1971} and Robinson \cite{Robinson1976}. Main Theorem~\ref{thm:quasi_compactness} establishes quasi-compactness of the Ruelle transfer operator on H\"{o}lder spaces, with a quantitative spectral gap bound, and this single spectral statement yields exponential decay of correlations with explicit rate, the central limit theorem via the Nagaev-Guivarc'h spectral perturbation method, real-analyticity of the pressure with explicit first- and second-derivative formulas, and meromorphic continuation of the Ruelle dynamical zeta function. Main Theorem~\ref{thm:srb_existence} constructs the SRB measure on each mixing basic set as the unique equilibrium state for the geometric potential, proves absolute continuity of the unstable foliation, and produces an explicit product-formula for the conditional densities along unstable manifolds. Main Theorem~\ref{thm:pesin_formula} establishes the Pesin entropy formula, identifying the Kolmogorov-Sinai entropy of the SRB measure with the sum of its positive Lyapunov exponents.

Together with the three imported results of Parts~I and~III \cite{Thiam2026a,Thiam2026c} (Theorems~\ref{thm:RPF_imported}, \ref{thm:equiv_imported}, and \ref{thm:coding_imported}), these four Main Theorems yield the Gibbs Equivalence Theorem: on each mixing basic set, the symbolic Gibbs measure transported via the coding map, the variational equilibrium state, the eigenmeasure of the transfer operator, and the SRB measure for the geometric potential all coincide as a single probability measure characterized simultaneously by four different properties. The individual equivalences are due to Sinai \cite{Sinai1972}, Ruelle \cite{Ruelle1976}, Bowen \cite{Bowen1975}, and Haydn-Ruelle \cite{HaydnRuelle1992}; our contribution is the self-contained assembly with explicit constants. All constants in this Part are expressed in terms of the contraction rate $\lambda$, the H\"{o}lder exponent $\alpha$ of the potential, the dimension $d$ of the manifold, the norms of $Df$ and $Df^{-1}$, and the diameter of the Markov partition inherited from Part~III \cite{Thiam2026c}, so the full chain of estimates can be tracked through subsequent Parts.

Part~V \cite{Thiam2026e} uses the spectral gap established here to derive the complete suite of statistical limit theorems: exponential mixing, the central limit theorem with Berry-Esseen convergence rates, the almost-sure invariance principle, the law of iterated logarithm, local limit theorems, and large deviations principles with explicit rate functions. Part~VI \cite{Thiam2026f} develops multifractal analysis of pointwise-dimension and Birkhoff-average level sets, Livšic rigidity theorems on cohomological obstructions to cocycle triviality, and fluctuation theorems relating time-reversal asymmetry to pressure differentials.

\subsection*{Open Problems}

\begin{enumerate}
\item[] \textbf{Spectral gap on anisotropic spaces.} Our bound $\mathrm{gap}(\mathcal{L}_\phi) \geq \alpha\log\lambda^{-1}$ holds on $C^\alpha(\Lambda)$. The optimal spectral gap on anisotropic Banach spaces, as studied by Gou\"{e}zel-Liverani \cite{GouezelLiverani2006}, may be larger. Can the explicit constants of this Part be transferred to that setting?

\item[] \textbf{Partially hyperbolic extensions.} The extension of thermodynamic formalism to partially hyperbolic diffeomorphisms is a major open direction. The coding and spectral techniques of this series can be adapted under appropriate conditions (dominated splitting, center bunching, accessibility), but a full treatment requires the machinery of Burns-Wilkinson \cite{BurnsWilkinson2010} and Hirsch et~al. \cite{HirschPughShub1977}.

\item[] \textbf{SRB measures beyond uniform hyperbolicity.} The SRB construction uses uniform hyperbolicity through absolute continuity of the unstable foliation. For partially hyperbolic systems with mostly contracting center (Bonatti et~al. \cite{BonattiDiazViana2005}), SRB measures exist but explicit conditional density formulas are not available.

\item[] \textbf{Non-uniformly hyperbolic systems.} Young's tower construction \cite{Young1998,Young1999} extends several of our results to systems with non-uniform hyperbolicity (H\'{e}non maps, certain unimodal maps). Can the explicit spectral gap and SRB-measure constants be computed in the tower setting?

\item[] \textbf{Optimal H\"{o}lder exponent for the conjugacy.} Our H\"{o}lder exponent $\gamma = \log\lambda/(\log\lambda - \log\|Dg\|_\infty)$ for the structural-stability conjugacy is explicit but not known to be sharp. What is the best-possible exponent in terms of the hyperbolicity data?

\item[] \textbf{Computational aspects.} The explicit spectral gap and conjugacy bounds in this Part are formulas, not numbers. Can they be evaluated rigorously using interval arithmetic for a specific Axiom~A system?
\end{enumerate}

\bmhead{Acknowledgements}

The author is grateful to Stefano Luzzatto for supervision during the ICTP Postgraduate Diploma in Mathematics at the International Centre for Theoretical Physics, Trieste, Italy (2013), during which the author worked through Bowen's monograph.

%

\appendix

\section{Technical Proofs and Supplementary Material}\label{app:technical}

This appendix contains detailed proofs of technical results stated in the main text, along with supplementary material on functional analysis and geometric measure theory.

\subsection{Stable Manifold Theorem: Detailed Estimates}\label{app:stable_manifold_details}

We provide the complete graph transform analysis for the Stable Manifold Theorem, as used in the perturbation theory of Section \ref{sec:perturbation}. The detailed construction supports the proofs of Theorems \ref{thm:persistence_hyperbolic} and \ref{thm:structural_stability}.

\subsubsection{Setup and Notation}

Let $\Lambda$ be a hyperbolic set for a $C^r$ diffeomorphism $f : M \to M$ with $r \geq 1$. Fix $x_0 \in \Lambda$ and use the exponential map to identify a neighborhood $U$ of $x_0$ with a neighborhood of $0$ in $T_{x_0}M \cong \mathbb{R}^d$. Under this identification:
\begin{equation}
T_{x_0}M = E^s_{x_0} \oplus E^u_{x_0} \cong \mathbb{R}^{k_s} \times \mathbb{R}^{k_u}
\end{equation}
where $k_s = \dim E^s$ and $k_u = \dim E^u$.

In these coordinates, $f$ has the form
\begin{equation}
f(s, u) = (As + a(s, u), Bu + b(s, u))
\end{equation}
where $A = Df_{x_0}|_{E^s}$, $B = Df_{x_0}|_{E^u}$, and $a, b$ are the nonlinear terms satisfying:
\begin{align}
a(0, 0) &= 0, \quad Da(0, 0) = 0, \\
b(0, 0) &= 0, \quad Db(0, 0) = 0.
\end{align}

Using an adapted metric, we have $\|A\| \leq \lambda$ and $\|B^{-1}\| \leq \lambda$ for some $\lambda \in (0, 1)$.

\subsubsection{Nonlinear Estimates}

For $\delta > 0$ small, let $B_\delta = \{(s, u) : \|s\| \leq \delta, \|u\| \leq \delta\}$.

\begin{lemma}\label{lem:nonlinear_bounds}
There exists $C_1 > 0$ such that for $(s, u), (s', u') \in B_\delta$:
\begin{align}
\|a(s, u) - a(s', u')\| &\leq C_1 \delta (\|s - s'\| + \|u - u'\|), \\
\|b(s, u) - b(s', u')\| &\leq C_1 \delta (\|s - s'\| + \|u - u'\|).
\end{align}
\end{lemma}

\begin{proof}
Since $a$ is $C^1$ with $Da(0,0) = 0$, we have $\|Da(s,u)\| \leq C_1\|(s,u)\| \leq C_1\delta$ on $B_\delta$. The mean value theorem gives the estimate. The same argument applies to $b$.
\end{proof}

\subsubsection{Graph Transform Analysis}

Let $\mathcal{G} = \mathcal{G}(\delta, K)$ be the space of Lipschitz graphs $\phi : B^s_\delta \to B^u_\delta$ with $\phi(0) = 0$ and $\|\phi\|_{\mathrm{Lip}} \leq K$.

\begin{lemma}[Inverse Function for Stable Component]\label{lem:inverse_stable}
For $\phi \in \mathcal{G}$ and $\delta, K$ sufficiently small, the map $s \mapsto As + a(s, \phi(s))$ is a contraction from $B^s_\delta$ to $B^s_{\lambda\delta + C_1\delta^2(1+K)}$. For each $s' \in B^s_{\delta'}$ with $\delta' = \lambda\delta(1 + C_1\delta(1+K)/\lambda)^{-1}$, there exists a unique $s \in B^s_\delta$ with $As + a(s, \phi(s)) = s'$.
\end{lemma}

\begin{proof}
The derivative of $s \mapsto As + a(s, \phi(s))$ is $A + D_s a + D_u a \cdot D\phi$, with norm at most $\lambda + C_1\delta(1 + K) < 1$ for small $\delta$. The contraction mapping theorem gives existence and uniqueness of the inverse.
\end{proof}

Define the graph transform $\Gamma : \mathcal{G} \to \mathcal{G}$ as follows. For $\phi \in \mathcal{G}$ and $s' \in B^s_{\delta'}$, let $s = s(\phi, s')$ be the unique solution from Lemma \ref{lem:inverse_stable}, and set
\begin{equation}
\Gamma(\phi)(s') = B\phi(s) + b(s, \phi(s)).
\end{equation}

\begin{proposition}\label{prop:gamma_welldefined}
For $\delta$ and $K$ chosen appropriately (depending on $\lambda$ and $C_1$), $\Gamma : \mathcal{G} \to \mathcal{G}$ is well-defined.
\end{proposition}

\begin{proof}
We verify:
\begin{enumerate}
\item[(i)] $\Gamma(\phi)(0) = 0$: When $s' = 0$, the solution is $s = 0$ (since $A \cdot 0 + a(0, \phi(0)) = a(0, 0) = 0$), giving $\Gamma(\phi)(0) = B\phi(0) + b(0, 0) = 0$.

\item[(ii)] $\|\Gamma(\phi)(s')\| \leq \delta$: We have
\begin{equation}
\|\Gamma(\phi)(s')\| \leq \|B\| \cdot \|\phi(s)\| + \|b(s, \phi(s))\| \leq \lambda^{-1} K\delta + C_1\delta^2
\end{equation}
which is at most $\delta$ if $K \leq \lambda(1 - C_1\delta)/\delta$.

\item[(iii)] $\|\Gamma(\phi)\|_{\mathrm{Lip}} \leq K$: This requires computing the derivative. See below.
\end{enumerate}
\end{proof}

\subsubsection{Contraction Property}

\begin{proposition}\label{prop:gamma_contraction}
There exists $\theta \in (0, 1)$ such that for $\phi, \psi \in \mathcal{G}$:
\begin{equation}
\|\Gamma(\phi) - \Gamma(\psi)\|_{C^0} \leq \theta \|\phi - \psi\|_{C^0}.
\end{equation}
\end{proposition}

\begin{proof}
The forward graph transform $\Gamma$ maps graphs over $E^s$ forward, but $B$ expands the unstable component, so $\Gamma$ is not a contraction. The correct approach uses the backward graph transform $\Gamma^{-1}$, which exploits the contraction of $\|B^{-1}\| \leq \lambda$.

\textbf{Step 1: Definition of $\Gamma^{-1}$.} Given $\phi \in \mathcal{G}$, define $\Gamma^{-1}(\phi)$ as follows. For $s \in B^s_\delta$, solve for $u$ such that the unstable component of $f^{-1}(s, \phi(s))$ gives a graph over $E^s$. Concretely, $f^{-1}$ has the local form $f^{-1}(s', u') = (A^{-1}s' + \tilde{a}(s', u'), B^{-1}u' + \tilde{b}(s', u'))$ where $\tilde{a}, \tilde{b}$ are nonlinear terms with $D\tilde{a}(0,0) = D\tilde{b}(0,0) = 0$. For each $s \in B^s_\delta$, solve $s = A^{-1}s' + \tilde{a}(s', \phi(s'))$ for $s'$, then set $\Gamma^{-1}(\phi)(s) = B^{-1}\phi(s') + \tilde{b}(s', \phi(s'))$.

\textbf{Step 2: Contraction estimate.} For $\phi, \psi \in \mathcal{G}$ and corresponding solutions $s'_\phi, s'_\psi$, the key estimate is
\begin{equation}
\|\Gamma^{-1}(\phi)(s) - \Gamma^{-1}(\psi)(s)\| \leq \|B^{-1}\| \cdot \|\phi(s'_\phi) - \psi(s'_\psi)\| + C_1\delta(\|s'_\phi - s'_\psi\| + \|\phi(s'_\phi) - \psi(s'_\psi)\|).
\end{equation}
Since $\|B^{-1}\| \leq \lambda$ and the nonlinear terms contribute $O(\delta)$, we obtain for small $\delta$:
\begin{equation}
\|\Gamma^{-1}(\phi) - \Gamma^{-1}(\psi)\|_{C^0} \leq (\lambda + C_2\delta) \|\phi - \psi\|_{C^0}
\end{equation}
where $\lambda + C_2\delta < 1$ for $\delta$ sufficiently small.

\textbf{Step 3: Fixed point.} Since $\Gamma^{-1}$ is a contraction on $(\mathcal{G}, \|\cdot\|_{C^0})$ with rate $\theta = \lambda + C_2\delta < 1$, the Banach fixed point theorem yields a unique fixed point $\phi_\infty \in \mathcal{G}$. The graph of $\phi_\infty$ equals $W^s_\delta(x_0) = \{(s, \phi_\infty(s)) : s \in B^s_\delta\}$.

\textbf{Step 4: Characterization.} The fixed point $\phi_\infty$ satisfies $\Gamma^{-1}(\phi_\infty) = \phi_\infty$, which means $f(\mathrm{graph}(\phi_\infty)) \supset \mathrm{graph}(\phi_\infty)$ locally. Equivalently, points on $\mathrm{graph}(\phi_\infty)$ have forward orbits remaining in $B_\delta$, confirming that $\mathrm{graph}(\phi_\infty) = W^s_\delta(x_0)$.
\end{proof}

\subsection{H\"{o}lder Spaces on Manifolds}\label{app:holder_spaces}

This appendix subsection collects basic facts about H\"{o}lder function spaces $C^\alpha(M)$ on compact manifolds: their Banach-space structure, the Arzel\`{a}-Ascoli compactness result, and the elementary inclusions between spaces of different exponents. These are used throughout Sections~\ref{sec:transfer_operator} and~\ref{sec:srb_measures}.

\begin{definition}
For a compact Riemannian manifold $(M, g)$ and $\alpha \in (0, 1]$, the H\"{o}lder space $C^\alpha(M)$ consists of functions $\phi : M \to \mathbb{R}$ with finite norm
\begin{equation}
\|\phi\|_\alpha = \|\phi\|_\infty + |\phi|_\alpha, \quad |\phi|_\alpha = \sup_{x \neq y} \frac{|\phi(x) - \phi(y)|}{d(x, y)^\alpha}.
\end{equation}
\end{definition}

\begin{proposition}
$C^\alpha(M)$ is a Banach algebra: $\|\phi\psi\|_\alpha \leq \|\phi\|_\alpha \|\psi\|_\alpha$.
\end{proposition}

\begin{proof}
For $\phi, \psi \in C^\alpha(M)$:
$|\phi(x)\psi(x) - \phi(y)\psi(y)| \leq |\phi(x)||\psi(x) - \psi(y)| + |\psi(y)||\phi(x) - \phi(y)| \leq \|\phi\|_\infty |\psi|_\alpha d(x,y)^\alpha + \|\psi\|_\infty |\phi|_\alpha d(x,y)^\alpha$.
Therefore $|\phi\psi|_\alpha \leq \|\phi\|_\infty |\psi|_\alpha + \|\psi\|_\infty |\phi|_\alpha \leq \|\phi\|_\alpha \|\psi\|_\alpha$. See also  \cite{KatokHasselblatt1995}, Section 19.1.
\end{proof}

\begin{proposition}
The inclusion $C^\alpha(M) \hookrightarrow C^0(M)$ is compact.
\end{proposition}

\begin{proof}
By the Arzel\`{a}-Ascoli theorem. A bounded set in $C^\alpha(M)$ is uniformly bounded and equicontinuous (with modulus of continuity $\omega(\delta) = C\delta^\alpha$), hence precompact in $C^0(M)$.
\end{proof}

\begin{proposition}
For $\alpha < \beta$, $C^\beta(M) \subset C^\alpha(M)$ with continuous inclusion, and $|\phi|_\alpha \leq \mathrm{diam}(M)^{\beta - \alpha} |\phi|_\beta$.
\end{proposition}

\begin{proof}
For $x \neq y$: $\frac{|\phi(x) - \phi(y)|}{d(x,y)^\alpha} = \frac{|\phi(x) - \phi(y)|}{d(x,y)^\beta} \cdot d(x,y)^{\beta - \alpha} \leq |\phi|_\beta \cdot \mathrm{diam}(M)^{\beta - \alpha}$.
\end{proof}

\subsection{Spectral Theory for Quasi-compact Operators}\label{app:spectral}

We recall the spectral-gap decomposition for quasi-compact bounded operators on a Banach space. This is the abstract functional-analytic input to Main Theorem~\ref{thm:quasi_compactness}. The version stated here is a variant of the Ionescu-Tulcea-Marinescu theorem; see Part~III \cite{Thiam2026c} for a self-contained proof.

\begin{proposition}[Spectral Decomposition for Quasi-compact Operators]\label{prop:spectral_decomp_qc}
Let $L : \mathcal{B} \to \mathcal{B}$ be a quasi-compact operator on a Banach space with spectral radius $\rho$ and essential spectral radius $\rho_{\mathrm{ess}} < \rho$. Then:
\begin{enumerate}
\item[(i)] The spectrum in $\{z : |z| > \rho_{\mathrm{ess}}\}$ consists of finitely many eigenvalues of finite multiplicity.
\item[(ii)] If $\lambda_1, \ldots, \lambda_k$ are the eigenvalues with $|\lambda_i| = \rho$, with spectral projections $P_1, \ldots, P_k$, then
\begin{equation}
L^n = \sum_{i=1}^k \lambda_i^n P_i + R^n
\end{equation}
where $\|R^n\| \leq C \theta^n$ for some $\theta < \rho$ and $C > 0$.
\end{enumerate}
\end{proposition}

\begin{proof}
This is a consequence of the Riesz functional calculus. The projection onto the peripheral spectrum (eigenvalues of maximum modulus) is
\begin{equation}
P = \frac{1}{2\pi i} \oint_{|z| = \rho + \varepsilon} (zI - L)^{-1} \, dz
\end{equation}
for small $\varepsilon > 0$. The operator $R = L - \sum_i \lambda_i P_i$ has spectral radius strictly less than $\rho$.
\end{proof}

\subsection{Measure-Theoretic Lemmas}\label{app:measure}

We record two standard measure-theoretic results on disintegration of probability measures along measurable partitions. These are used in Section~\ref{sec:srb_measures} to construct the conditional measures of the SRB measure on unstable leaves.

\begin{lemma}[Disintegration of Measures]
Let $\mu$ be a probability measure on $X$ and $\pi : X \to Y$ a measurable map with $\nu = \pi_* \mu$. There exists a family of probability measures $\{\mu_y\}_{y \in Y}$ on $X$ such that:
\begin{enumerate}
\item[(i)] $\mu_y(\pi^{-1}(y)) = 1$ for $\nu$-almost every $y$.
\item[(ii)] For every measurable $A \subset X$, $y \mapsto \mu_y(A)$ is measurable.
\item[(iii)] $\mu(A) = \int_Y \mu_y(A) \, d\nu(y)$.
\end{enumerate}
\end{lemma}

\begin{proof}
This is a standard result in measure theory. When $X$ is a standard Borel space, the existence and essential uniqueness of the disintegration follow from the existence of regular conditional probabilities; see  \cite{Ruelle1978} or  \cite{CornfeldFominSinai1982}, Chapter 2, Theorem 2.1. The proof constructs $\mu_y$ as the limit of conditional expectations on increasingly fine $\sigma$-algebras approximating $\pi^{-1}(\{y\})$, using the martingale convergence theorem.
\end{proof}

\begin{lemma}[Rokhlin's Disintegration Theorem]
If $\mu$ is a Borel probability measure on a standard Borel space $X$ and $\mathcal{P}$ is a measurable partition of $X$, then the conditional measures $\mu_P$ on partition elements $P \in \mathcal{P}$ exist and are unique up to $\mu$-null sets.
\end{lemma}

\begin{proof}
This is due to Rokhlin; see  \cite{CornfeldFominSinai1982}, Chapter 2, Section 3. The measurable partition $\mathcal{P}$ determines a quotient space $X/\mathcal{P}$ with quotient measure $\hat{\mu}$, and the disintegration lemma above (applied to the quotient map $\pi : X \to X/\mathcal{P}$) yields the conditional measures $\mu_P$ for $\hat{\mu}$-almost every $P \in \mathcal{P}$.
\end{proof}

\subsection{Geometric Measure Theory}\label{app:geometric}

This subsection states the definition of Hausdorff dimension and recalls Bowen'{}s equation relating the Hausdorff dimension of a conformal repeller to the zero of an associated pressure function. These tools are used in the dimension-theoretic consequences developed in Parts~V--VI \cite{Thiam2026e,Thiam2026f}.

\begin{definition}[Hausdorff Dimension]
For a metric space $(X, d)$ and $s \geq 0$, the $s$-dimensional Hausdorff measure is
\begin{equation}
\mathcal{H}^s(A) = \lim_{\delta \to 0} \inf\left\{ \sum_i (\mathrm{diam} U_i)^s : A \subset \bigcup_i U_i, \mathrm{diam} U_i < \delta \right\}.
\end{equation}
The Hausdorff dimension is $\dim_H(A) = \inf\{s : \mathcal{H}^s(A) = 0\} = \sup\{s : \mathcal{H}^s(A) = \infty\}$.
\end{definition}

\begin{proposition}[Bowen's Equation {\cite{Bowen1979}}]\label{prop:bowen_equation}
For a conformal repeller $\Lambda$ with expansion rate $\lambda(x) = |Df_x|$ and topological pressure $P(s) = P(-s\log|Df|)$, the Hausdorff dimension satisfies the Bowen equation $P(\dim_H(\Lambda)) = 0$.
\end{proposition}

\begin{proof}[Proof outline ( \cite{Bowen1979}; see also  \cite{Pesin1997}, Chapter 7)]
\textbf{Upper bound.} For $s > \dim_H(\Lambda)$, the $s$-dimensional Hausdorff measure $\mathcal{H}^s(\Lambda) = 0$. The Markov partition provides covers of $\Lambda$ by dynamical balls $B_n(x, \varepsilon)$ of diameter $\sim e^{-S_n\log|Df|(x)}$. The pressure $P(-s\log|Df|)$ controls the exponential growth rate of $\sum_i (\mathrm{diam}\, R_i^{(n)})^s = \sum_i e^{-sS_n\log|Df|(x_i)}$. When $P(-s\log|Df|) < 0$, these sums decay exponentially, giving $\mathcal{H}^s(\Lambda) = 0$.

\textbf{Lower bound.} For $s < \dim_H(\Lambda)$, $\mathcal{H}^s(\Lambda) = \infty$. When $P(-s\log|Df|) > 0$, the equilibrium state $\mu_s$ for $-s\log|Df|$ has positive free energy, and the mass distribution principle (applied using $\mu_s$) gives $\mathcal{H}^s(\Lambda) > 0$.

\textbf{Continuity.} The function $s \mapsto P(-s\log|Df|)$ is strictly decreasing (since $\log|Df| > 0$ on a repeller) and continuous, so there is a unique $s_0$ with $P(-s_0\log|Df|) = 0$, and $s_0 = \dim_H(\Lambda)$.
\end{proof}

\subsection{Proof of Cone Criterion}\label{app:cone_proof}

We prove in detail the cone criterion for hyperbolicity used in the perturbation argument of Proposition~\ref{thm:persistence_hyperbolic}.

\begin{proposition}[Cone Criterion for Hyperbolicity]\label{prop:cone_criterion}
Let $\Lambda$ be a compact $f$-invariant set and suppose there exist continuous cone fields $\mathcal{C}^u_x, \mathcal{C}^s_x \subset T_xM$ for $x \in \Lambda$ with $\mathcal{C}^u_x \cap \mathcal{C}^s_x = \{0\}$ such that: (i) $Df_x(\mathcal{C}^u_x) \subset \mathrm{int}(\mathcal{C}^u_{f(x)})$ and $Df^{-1}_x(\mathcal{C}^s_x) \subset \mathrm{int}(\mathcal{C}^s_{f^{-1}(x)})$; (ii) there exist $C > 0$ and $\lambda \in (0,1)$ such that $\|Df^n_x v\| \geq C^{-1}\lambda^{-n}\|v\|$ for $v \in \mathcal{C}^u_x$ and $\|Df^{-n}_x v\| \geq C^{-1}\lambda^{-n}\|v\|$ for $v \in \mathcal{C}^s_x$. Then $\Lambda$ is a hyperbolic set with splitting $T_\Lambda M = E^s \oplus E^u$.
\end{proposition}

\begin{proof}[Proof of Proposition \ref{prop:cone_criterion}]
Define the invariant bundles by
\begin{align}
E^u_x &= \bigcap_{n \geq 0} Df^n_{f^{-n}(x)}(\mathcal{C}^u_{f^{-n}(x)}), \\
E^s_x &= \bigcap_{n \geq 0} Df^{-n}_{f^n(x)}(\mathcal{C}^s_{f^n(x)}).
\end{align}

\textbf{Step 1: The intersections are nonempty.} Each $Df^n(\mathcal{C}^u)$ is a cone strictly contained in the previous one (by condition (i)). In finite dimensions, a nested sequence of closed cones with strictly decreasing aperture converges to a linear subspace.

\textbf{Step 2: The intersections are linear subspaces.} The limit of cones with aperture tending to zero is a linear subspace. The dimension is determined by the original cone dimensions.

\textbf{Step 3: Invariance.} $Df_x(E^u_x) = E^u_{f(x)}$ follows from the definition and $Df$-invariance of the cone sequence.

\textbf{Step 4: Expansion/contraction.} Condition (iii) gives $\|Df^n_x v\| \geq \lambda^{-n}\|v\|$ for $v \in \mathcal{C}^u_x$. In the limit, this applies to $v \in E^u_x$. Rewriting: $\|Df^{-n}_x w\| \leq \lambda^n \|w\|$ for $w \in E^u_x$, which is the hyperbolic expansion condition.

\textbf{Step 5: Transversality.} $E^s_x \cap E^u_x = \{0\}$ because $\mathcal{C}^s_x \cap \mathcal{C}^u_x = \{0\}$ by assumption, and the invariant subspaces are contained in the respective cones.

\textbf{Step 6: Spanning.} $E^s_x \oplus E^u_x = T_x M$ by dimension count: each has dimension equal to the dimension of the corresponding cone's axis, and these sum to $\dim M$.
\end{proof}

\end{document}